\documentclass[a4paper,10pt]{amsart}
\usepackage{times}
\usepackage{amsmath,amsthm,amssymb,enumerate}
\usepackage{epsfig}
\usepackage{amssymb}
\usepackage{amsmath}
\usepackage{amssymb}
\usepackage{amsmath,amsthm}
\usepackage{eqparbox}
\usepackage[latin1]{inputenc}
\usepackage{bm}
\usepackage{bbm}
\usepackage{esint}
\usepackage{tikz-cd}
\usetikzlibrary{decorations.pathreplacing}
\usepackage{amsfonts}
\usepackage{amsxtra}
\usepackage{euscript,mathrsfs}
\usepackage{color}
\usepackage{verbatim}
\usepackage[left=4cm,right=4cm,top=3.5cm,bottom=3cm]{geometry}
\usepackage[colorlinks=true, linktocpage=true, linkcolor=red!70!black, citecolor=green!50!black]{hyperref}
\allowdisplaybreaks

\usepackage{tikz-cd}
\usetikzlibrary{decorations.pathreplacing}

\usepackage{enumitem}
\setenumerate{label={\rm (\alph{*})}}

\usepackage{amsfonts}
\usepackage{amsxtra}

\numberwithin{equation}{section}

\usepackage{xcolor}
\colorlet{ColorPink}{red!30}
\usepackage{graphicx}

\newcommand{\R}{\mathbb R}

\newcommand{\dif}{\mathrm{d}}
\newcommand{\Reg}{\mathrm{Reg}}

\renewcommand{\dif}{\operatorname{d}\!}
\newcommand{\lebe}{\operatorname{L}}
\newcommand{\sobo}{\operatorname{W}}

\newcommand{\locc}{\operatorname{loc}}

\newcommand{\hold}{\operatorname{C}}

\newcommand{\sg}{\varepsilon}

\newcommand{\ball}{\operatorname{B}}

\newcommand{\A}{\mathcal{A}}

\newcommand{\reg}{\mathrm{Reg}}
\newcommand{\sing}{\mathrm{Sing}}
\newcommand{\dashint}{\fint}

\allowdisplaybreaks

\theoremstyle{plain}

\newtheorem*{theorem*}{Theorem}
\newtheorem{theorem}{Theorem}[section]
\newtheorem{lemma}[theorem]{Lemma}

\newtheorem{corollary}[theorem]{Corollary}

\theoremstyle{remark}

\newtheorem{remark}[theorem]{Remark}

\def \red{\textcolor{red}}
\def \green{\textcolor{teal}}

\begin{document}


\title[$\mathcal{A}$-harmonic approximation and partial regularity]{$\mathcal{A}$-harmonic approximation \\ and partial regularity, revisited}
\author[M.~B\"{a}rlin]{Matthias B\"{a}rlin}
\author[F. Gmeineder]{Franz Gmeineder}
\address{M.B. \& F.G.: University of Konstanz, Department of Mathematics and Statistics, Universit\"{a}tsstra\ss e 10, 78464 Konstanz, Germany}
\email{matthias.baerlin@uni-konstanz.de}
\email{franz.gmeineder@uni-konstanz.de}
\author[C.~Irving]{Christopher Irving}
\address{C.I.: TU Dortmund University, Faculty of Mathematics, Vogelpothsweg 87, 44227 Dortmund, Germany}
\email{christopher.irving@tu-dortmund.de}
\author[J.~Kristensen]{Jan Kristensen}
\address{C.I. \& J.K.: University of Oxford, Mathematical Institute, Radcliffe Observatory Quarter, OX2 6HG, Oxford, United Kingdom}
\email{kristens@maths.ox.ac.uk}

\maketitle

\begin{abstract}
  We give a direct harmonic approximation lemma for local minima of quasiconvex multiple integrals that entails their $\hold^{1,\alpha}$ or $\hold^{\infty}$-partial regularity. Different from previous contributions, the method is fully direct and elementary, only hinging on the $\lebe^{p}$-theory for strongly elliptic linear systems and Sobolev's embedding theorem. Especially, no heavier tools such as Lipschitz truncations are required.
\end{abstract}

\section{Introduction}
In this paper we provide an alternative approach to arrive at the gradient partial regularity of local minimisers of the basic multiple integral 
\begin{align}\label{eq:mainfct}
\mathscr{F}[u;\Omega]:=\int_{\Omega}F(\nabla u)\dif x,\qquad u\colon \Omega \to \R^N,
\end{align}
for variational integrands $F\colon\R^{N\times n}\to\R$ and open and bounded subsets $\Omega\subset\R^N$. It is well-known that even under strong smoothness and convexity assumptions on $F$, local minima need not be of class $\hold^{\infty}$ in the entire $\Omega$, but merely \emph{partially $\hold^{\infty}$-regular}: There exists an open subset $\Omega_{0}$ of $\Omega$ such that $u$ is of class $\hold^{\infty}$ in $\Omega_{0}$ (see \cite{AcerbiFusco,Beck,CFM,CP,DLSV,Duzaar1,Duzaar2,DuzaarMingione,Evans,FH,GM,Giusti,GK1,Hamburger} for a non-exhaustive list of contributions). To state our result conveniently, we shall make the following assumptions on $F$: There exist $L>0$ and, for each $m>0$, $\ell_{m}>0$ such that 
\begin{align}\label{eq:pgrowth}
\begin{cases}
|F(z)|\leq L(1+|z|^{p})&\text{for all}\;z\in\R^{N\times n}, \\ 
z\mapsto F(z)-\ell_{m} V_{p}(z) &\text{is quasiconvex at all $z_{0}\in\R^{N\times n}$ with  $|z_{0}|\leq m$},\\
 F\in\hold^{2}(\R^{N\times n}).&
\end{cases}
\end{align}
Here, $V_{p}(z):=(1+|z|^{2})^{\frac{p}{2}}-1$, and we recall that a continuous integrand $F\colon\R^{N\times n}\to\R$ is called \emph{quasiconvex at $z_{0}\in\R^{N\times n}$} provided we have 
\begin{align}\label{eq:defquasiconvexity}
F(z_{0})\leq \dashint_{\ball_1 (0)}F(z_{0}+\nabla\varphi)\dif x\qquad\text{for all  $\varphi\in\sobo_{0}^{1,\infty}(\ball_1 (0);\R^N)$}, 
\end{align}
and simply \emph{quasiconvex} provided it is quasiconvex at every $z_{0}$. 
The first two conditions in~\eqref{eq:pgrowth} embody the fact that the variational integral~\eqref{eq:mainfct} is well-defined on subclasses of $\sobo^{1,p}(\Omega;\R^{N})$ by~$\eqref{eq:pgrowth}_{1}$, and  both coercive and lower semicontinuous for weak convergence on Dirichlet subclasses of $\sobo^{1,p}(\Omega;\R^{N})$ by $\eqref{eq:pgrowth}_{2}$ (cf. \cite{CK,Meyers}); the third one is a minimal regularity assumption necessary to carry out the linearisation procedure detailed below. 

In establishing partial regularity theorems, there are by now a wealth of methods, so for instance the \emph{blow-up method} or the \emph{method of $\mathcal{A}$-harmonic approximation}; see \textsc{Beck} \cite{Beck} or \textsc{Mingione} \cite{Mingione} for overviews. The overarching metaprinciple in the $\mathcal{A}$-harmonic approximation technique is the $\mathcal{A}$\emph{-harmonic approximation} lemma.
Broadly speaking, it asserts that if a map is almost $\mathcal{A}$-harmonic in a quantifiable way for a suitable bilinear form $\mathcal{A}\colon\R^{N\times n}\times\R^{N\times n}\to \R$ (to be tacitly identified with its matrix representative), then it is \emph{close} to an $\mathcal{A}$-harmonic map. As usual, the $\mathcal{A}$-harmonicity of a map $u\in\sobo_{\locc}^{1,1}(\Omega;\R^{N})$ means that we have 
\begin{align}\label{eq:Aharmonic}
\int_{\Omega}\mathcal{A}[\nabla u,\nabla\varphi]\dif x := \int_{\Omega}\langle \mathcal{A}\nabla u,\nabla\varphi\rangle\dif x = 0\quad\text{for all}\;\varphi\in\sobo_{0}^{1,\infty}(\Omega;\R^N).
\end{align}

Typically, the aforementioned closeness is measured by a $V$-function type distance, e.g. see~\eqref{eq:Vfunction} below for potential choices of such auxiliary $V$-functions. By now, there are various techniques to arrive at such approximation lemmas, either by \emph{indirect means} (cf.~the approaches by~\textsc{De Giorgi} \cite{DeGiorgi} and~\textsc{Simon} \cite{Simon1,Simon2} for $\mathcal{A}=\mathrm{Id}$ or \textsc{Duzaar} et al. \cite{Duzaar1,Duzaar2,DuzaarMingione} in the more general context of Legendre-Hadamard elliptic bilinear forms) or \emph{direct means}, using the technically more sophisticated method of Lipschitz truncations (cf.~\textsc{Diening} et al.~\cite{DieningCruzzi,DLSV}).  

In this paper, we present a construction of $\mathcal{A}$-harmonic comparison maps that is \emph{fully direct} and does not utilise heavier tools such as Lipschitz truncations, yet yields the necessary estimates to conclude  partial regularity of minima for a wealth of variational problems. The proof itself only uses elementary embeddings for Sobolev spaces and routine results on the solvability of Legendre-Hadamard elliptic systems. The key approximation theorem of this paper is stated most conveniently for Legendre-Hadamard elliptic integrands, the connection of which to partial regularity proofs being discused below. More precisely, we assume slightly more generally than in \eqref{eq:pgrowth} that $F\colon\R^{N\times n}\to\R$ satisfies the following: There exist $L>0$ and, for each $m>0$, a number $\ell_{m}>0$ such that 
\begin{align}\label{eq:pgrowth1}
\begin{cases}
|F(z)|\leq L(1+|z|^{p})&\text{for all}\;z\in\R^{N\times n}, \\ 
z\mapsto F(z)-\ell_{m} V_{p}(z) &\text{is rank-one-convex at all $z_{0}\in\R^{N\times n}$ with  $|z_{0}|\leq m$},\\
 F\in\hold^{2}(\R^{N\times n}).&
 \end{cases}
\end{align}
As usual, rank-one convexity of an integrand $G\colon\R^{N\times n}\to\R$ at some $z_{0}$ means that for all $a\in\R^{N},b\in\R^{n}$ the function $\R\ni t \mapsto G(z_{0}+ta\otimes b)$ is convex. 
Our main theorem then reads as follows:
\begin{theorem*}
Let $F\colon\R^{N\times n}\to\R$ satisfy~\eqref{eq:pgrowth1} and $\ball_{r}(x_{0})\Subset\Omega$ be an open ball. For any $m>0$ and all affine-linear maps $a\colon\R^{n}\to\R^{N}$ with $|\nabla a|\leq m$ there  exist a constant $c=c(n,N,p,m,F)>0$ and a continuous function $H\colon\R_{\geq 0}\to\R_{\geq 0}$ depending on $n,p,m,F$ such that
\begin{align}\label{eq:vanishatzero}
\lim_{t\searrow 0}\frac{H(t)}{t}=0
\end{align}
with the following property: Whenever $u\in\sobo_{\locc}^{1,p}(\Omega;\R^{N})$ is an \emph{extremiser} of~\eqref{eq:mainfct} subject to~\eqref{eq:pgrowth1}, so satisfies \begin{align}\label{eq:extremiser}
\int_{\Omega}\langle F'(\nabla u),\nabla\varphi\rangle\dif x = 0\qquad\text{for all}\;\varphi\in\hold_{c}^{\infty}(\Omega;\R^{N}),
\end{align}
then we have 
\begin{align}\label{eq:AharmonicSharp}
\dashint_{\ball_{r}(x_{0})}V_{p}\Big(\frac{\widetilde{u}-\widetilde{h}}{r}\Big)\dif x \leq cH\Big(\dashint_{\ball_{r}(x_{0})}V_{p}(D\widetilde{u})\dif x\Big), 
\end{align}
where $\widetilde{u}:=u-a$ and $\widetilde{h}$ is the unique solution of the Legendre-Hadamard elliptic system 
\begin{align}\label{eq:systemMain}
\begin{cases} 
-\mathrm{div}(\mathcal{A}\nabla\widetilde{h}) = 0&\;\text{in}\;\ball_{r}(x_{0}), \\ 
\widetilde{h}=\widetilde{u}&\;\text{on}\;\partial\!\ball_{r}(x_{0})
\end{cases}
\end{align}
with $\A=F''(\nabla a)$. Here, we have set for an arbitrary normed space $(X,|\cdot|)$
\begin{align}\label{eq:Vfunction}
V_{p}(z):=\big(1+|z|^{2}\big)^{\frac{p}{2}}-1,\qquad z\in X. 
\end{align}
\end{theorem*} 
The previous theorem will be established in Section~\ref{sec:dirappro}, where we also address the dependence of the constant $c$ in Remark~\ref{rem:thm_constant}. In Section~\ref{sec:app}, we sketch how it enters partial regularity proofs for functionals of the form~\eqref{eq:mainfct} in Section \ref{sec:eps_reg}, and it is here where we will see the importance of condition~\eqref{eq:vanishatzero}; see Remark~\ref{rem:important} for more detail. We then conclude with improvements that incorporate uniformly controlled growth in Section \ref{sec:uniformlycontrolledgrowth}, where also links to Hausdorff dimension bounds for the singular set \`{a} la \textsc{Mingione} and the last named author \cite{KristensenMingione} are addressed.

Before we embark on these matters, let us briefly address the entering of the single notions of convexity in the partial regularity proof in the setting of \eqref{eq:pgrowth}. Aiming for an excess decay to access the \textsc{Campanato-Meyers} characterisation of H\"{o}lder continuity, one combines the above theorem with the Caccioppoli inequality of the second kind. Whereas the above theorem holds for extremisers of Legendre-Hadamard elliptic multiple integrals (which is in particular satisfied by strongly quasiconvex integrands), it is the Caccioppoli inequality that seems to require strong quasiconvexity. It is also at the Caccioppoli inequality of the second kind where   minimality instead of extremality is needed: In fact, otherwise the above argument would yield partial regularity of extremisers of strongly quasiconvex integrals, and the latter is well-known to be false by the fundamental work of \textsc{M\"{u}ller \& \v{S}ver\'{a}k} \cite{MullerSverak}. Hence, even though \eqref{eq:pgrowth} implies \eqref{eq:pgrowth1}, we will only address the partial regularity results in Section \ref{sec:app} for local minimisers of multiple integrals satisfying \eqref{eq:pgrowth}.

\section*{Notation}
Our notation is mostly standard, but we wish to emphasize some conventions: We shall frequently write $a\lesssim_{s}b$ to express that we have $a\leq c b$ for the two quantities $a,b$ with a constant $c=c(s)>0$. If $a\lesssim_{s} b$ and $b\lesssim_{s}a$, we write $a\simeq_{s}b$. Moreover, whenever $1<p<n$, we use  $p^{*}=\frac{np}{n-p}$ to denote its Sobolev conjugate exponent as usual. Finally, we denote the (euclidean or Hilbert-Schmidt) inner  product on $\R^{N}$ or on $\mathbb R^{N\times n}$ by $\langle \cdot, \cdot \rangle$, but no ambiguities will arise from this.
\section{Direct approximation by \texorpdfstring{$\mathcal{A}$}{A}-harmonic maps}\label{sec:dirappro}
\subsection{Four elementary lemmas} 
We now provide the proof of the above theorem. To this end, we merely require four elementary ingredients that we display next. We begin with 
\begin{lemma}\label{lem:Vaux}
Let $1<p<\infty$ and let $(X,\lvert\cdot\rvert)$ be a normed space. Then the following hold for all $z\in X$: 
\begin{enumerate}
\item\label{item:Vprops1} $V_{p}(z) \simeq_{p} (\mathbbm{1}_{\{|z|\leq 1\}}|z|^{2} +\mathbbm{1}_{\{|z|>1\}}|z|^{p})$.
\item\label{item:Vprops2} If $p\geq 2$, then $V_{p}(z) \simeq_{p} |z|^{2}+|z|^{p}$.
\item\label{item:Vprops2a} For any $m>0$ we have $V_{p}(z)\simeq_{p,m}|z|^{2}$ for all $z\in X$ with $|z|\leq m$.
\item\label{item:Vprops3} For any $1\leq r \leq \min\{2,p\}$, the function $t\mapsto V_{p}(t^{\frac{1}{r}})$ is convex. 
\item\label{item:Vprops4} $V_{p'}(V_{p}(z)/\lvert z\rvert) \lesssim_{p} V_{p}(z).$
\item\label{item:Vprops4A} $V_{p'}(V'_p (z))\lesssim_{p} V_p (z)$
\item\label{item:Vprops5} $V_{p}(\lambda z)\lesssim_{p} \lambda^{\max\{2,p\}}V_{p}(z)$ for all $\lambda\geq1$. 
\item\label{item:Vprops6} $V_{p}(\lambda z)\lesssim_{p} \lambda^{\min\{2,p\}}V_{p}(z)$ for all $\lambda\leq1$. 
\end{enumerate}
Moreover for $z,w \in X$ we additionally have
\begin{enumerate}\setcounter{enumi}{8}
  \item\label{item:Vprops7} $V_p(z+w) \lesssim_p V_p(z) + V_p(w)$ (subadditivity)
  \item\label{item:Vprops8}$\lvert z\rvert \lvert w\rvert \lesssim_p V_p(z) + V_{p'}(w)$ (Young's inequality).
\end{enumerate}
\end{lemma}
These assertions are well-known and can be found scattered throughout the literature, however we will outline the proofs for completeness.
\begin{proof}
Property~\ref{item:Vprops1} is a direct consequence of the definition of $V_{p}$, from which~\ref{item:Vprops2},~\ref{item:Vprops2a}, \ref{item:Vprops4} and ~\ref{item:Vprops4A} immediately follow.
By~\ref{item:Vprops1} we see that $V_p(z) \simeq_p \max\{\lvert z\rvert^p,\lvert z\rvert^2\}$ for $p>2$ and $V_p(z) \simeq_p \min\{\lvert z\rvert^p,\lvert z\rvert^2\}$ for $p\leq 2$, from where \ref{item:Vprops5} and \ref{item:Vprops6} follow. For \ref{item:Vprops7} we simply note that $V_p(z_1+z_2) \leq V_p(2z_1) + V_p(2z_2)$ and use \ref{item:Vprops5} with $\lambda=2.$
For~\ref{item:Vprops8} we will use~\ref{item:Vprops1} and the corresponding power-type estimates
\begin{equation}\label{eq:younginequality_power}
  \lvert z\rvert \lvert w\rvert \leq 
  \begin{cases} 
    \frac12 \lvert z\rvert^2 + \frac12 \lvert w\rvert^2,  \\
    \frac1p \lvert z\rvert^p + \frac1{p'} \lvert w\rvert^{p'},
  \end{cases}
\end{equation} 
which immediately yield the claim for $p\geq 2$ by \ref{item:Vprops2}. If $p<2$, we distinguish between the four different cases which arise: We use~$\eqref{eq:younginequality_power}_1$ if $\lvert z\rvert \geq 1$ and~$\eqref{eq:younginequality_power}_2$ if $\lvert z\rvert \leq 1.$
Finally for~\ref{item:Vprops3}, for $t>0$ we can compute
\begin{align*}
\frac{\dif^{2}}{\dif t^{2}}V_{p}(t^{\frac{1}{r}})= \frac{1-r}{r^{2}}t^{\frac{1-2r}{r}}V'_{p}(t^{\frac{1}{r}}) + \frac{1}{r^{2}}t^{\frac{2-2r}{r}}V''_{p}(t^{\frac{1}{r}}). 
\end{align*}
Substituting $\xi=t^{\frac{1}{r}}$ this expression is non-negative if and only if 
\begin{align}\label{eq:derivativesdtermine}
(r-1)\stackrel{!}{\leq} \frac{\xi V''_{p}(\xi)}{V'_{p}(\xi)}=\frac{(p-1)\xi^{2}+1}{\xi^{2}+1}\qquad \text{for all}\;\xi>0. 
\end{align}
It is easily seen that the ultimate expression is decreasing in $\xi>0$ for $p<2$, leading to its infimal value $(p-1)$ by sending $\xi\to\infty$, and non-decreasing in $\xi>0$ provided $1<p\leq 2$, leading to its infimal value $1$ by inserting $\xi=0$. Summarising, the right-hand side of \eqref{eq:derivativesdtermine} is bounded below by $\min\{p,2\}-1$, and so $t\mapsto V_{p}(t^{\frac{1}{r}})$ is convex if and and only if $r\leq \min\{2,p\}$ as claimed.
The proof is complete.
\end{proof}
Given $F\colon\R^{N\times n}\to\R$ of class $\hold^{1}$  and $w\in\R^{N\times n}$, we denote 
\begin{align}\label{eq:linearisedintegrands}
F_{w}(z):=F(w+z)-F(w)-\langle F'(w), z\rangle,\qquad z\in\R^{N\times n},
\end{align}
the corresponding \emph{shifted integrand}. We then have 
\begin{lemma}\label{lem:Faux}
  Let $1<p<\infty$ and suppose that $F\in\hold^{2}(\R^{N\times n})$ satisfies $\eqref{eq:pgrowth1}_1$ and  $\eqref{eq:pgrowth1}_2$. Then for each $m>0$ there exist constants $L_m >0$, $ \nu_m > 0,$ and a non-negative,  non-decreasing concave function $\omega_m \colon [0,\infty) \to [0,2]$ with  $\lim_{t\searrow 0}\omega_m(t)=0$ such that for all $w\in\R^{N\times n}$ with $|w|\leq m$ the following hold for all $z\in\R^{N\times n}$:
\begin{enumerate}
\item\label{item:Fprops1} $|F'_{w}(z)|\leq L_m (\mathbbm{1}_{\{|z|\leq 1\}}|z|+\mathbbm{1}_{\{|z|>1\}}|z|^{\max\{1,p-1\}})$. 
\item\label{item:Fprops2} $|F''_{w}(0)z - F'_{w}(z)|\leq L_m \omega_m(\lvert z\rvert) \lvert z\rvert^{\max\{1,p-1\}}$. 
\item\label{item:Fprops3} $\nu_m \lvert \xi\rvert^2 \lvert\eta\rvert^2 \leq \langle F''_w(0) (\xi\otimes\eta), (\xi\otimes\eta)\rangle \leq L_m \lvert \xi\rvert^2 \lvert\eta\rvert^2$ for all $\xi \in \R^N,$ $\eta \in \R^n.$
\end{enumerate}
Especially, \ref{item:Fprops1}--\ref{item:Fprops3} hold if $F\in\hold^{2}(\R^{N\times n})$ satisfies $\eqref{eq:pgrowth}_{1}$ and $\eqref{eq:pgrowth}_{2}$.
\end{lemma} 
\begin{proof} 
 Note that the rank-one convexity of $F$ and the growth condition~\eqref{eq:pgrowth} combine to the estimate $|F'(z)|\leq cL(1+|z|^{p-1})$; see, e.g., \cite[Lem.~5.2]{Giusti}.
 Also since $F$ is of class $\hold^2$ on $\mathbb{B}_{m+1}:=\{\xi\in\R^{N\times n}\colon\;|\xi|\leq m+1\}$, we have that $F''$ is bounded and uniformly continuous on $\mathbb{B}_{m+1}$; that is, there is $\Lambda_m>0$ and a modulus of continuity $\omega_m$ as in the statement of the lemma such that
 \begin{align}
   \lvert F''(\xi)\rvert &\leq \Lambda_m,\label{eq:dexter1}\\
   \lvert F''(\xi) - F''(\zeta)\rvert &\leq 2\Lambda_m\omega_m(\lvert \xi-\zeta\rvert)\label{eq:dexter2}
 \end{align}
 for all $\xi,\zeta \in \mathbb{B}_{m+1}.$ Moreover replacing $\omega_m$ by $\min\{2,2t+\omega_m(t)\},$ we can assume that $\omega_m(t)=2$ for all $t \geq 1.$ 
 Now to show~\ref{item:Fprops1}, noting if $|z|\leq 1$, we have $w+tz \in \mathbb{B}_{m+1}$ for all $t \in [0,1]$ and so
\begin{align}\label{eq:firstderivativebound}
\begin{split}
|F'_{w}(z)| & \leq |F'(w+z)-F'(w)| \leq \left\vert \int_{0}^{1}\frac{\dif}{\dif t}F'(w+tz)\dif t\right\vert \\ 
            & \leq \int_{0}^{1}|F''(w+tz)|\,|z|\dif t \leq \Lambda_m|z|, 
            \end{split}
\end{align}
whereas for $|z|\geq 1$ we have 
\begin{align*}
|F'_{w}(z)| \leq |F'(w+z)|+|F'(w)|\,|z|^{p-1} \leq c(n,N,p,m,L)|z|^{p-1}, 
\end{align*}
and so~\ref{item:Fprops1} follows. For~\ref{item:Fprops2}, we note that
\begin{align}\label{eq:joestriani}
  |F''_{w}(0)z - F'_{w}(z)| & = \left\vert \int_{0}^{1}(F''(w)-F''(w+tz))\dif t\right\vert\,|z| \stackrel{\eqref{eq:dexter2}}{\leq}  2\Lambda_m\omega_m(\lvert z\rvert)|z|
\end{align}
provided $|z|\leq 1$ and, again using the bound $|F'(z)|\leq L(1+|z|^{p-1})$, 
\begin{align*}
  |F''_{w}(0)z - F'_{w}(z)| \leq |F''(w)|\,|z|^{p-1} + c(m,L,p)|z| \leq c(m,L,p,\Lambda_m)|z|^{\max\{1,p-1\}} 
\end{align*}
for $|z|\geq 1$. Combining the two cases yields the claim.
Finally~\ref{item:Fprops3} follows by noting that for $\varphi \in \hold^{\infty}_c(\Omega;\R^N)$ the functional
\begin{equation*}
  \mathcal J(t) = \int_{\Omega} F(z_0+t\nabla\varphi) - F(z_0)- \ell_{m} V_p(z_0 + t\nabla\varphi) + \ell_{m} V_p(z_0) \dif x
\end{equation*} 
is non-negative and attains its minimum when $t=0.$
Hence $\mathcal J''(0) \geq 0,$ which implies the G\r{a}rding inequality
\begin{equation*}
  \int_{\Omega} \langle F''(z_0)\nabla\varphi,\nabla\varphi\rangle \dif x \geq  \ell_{m} \, V_p''(z_0) \int_{\Omega}\lvert\nabla\varphi\rvert^2 \dif x,
\end{equation*} 
from which the Legendre-Hadamard condition \ref{item:Fprops3} follows. Since quasiconvexity at $z_{0}$ implies rank-one convexity at $z_{0}$, the proof is complete. 
\end{proof} 
Third, we record a Poincar\'e-Sobolev inequality for $V$-type functionals. The exponents can be improved, but we only require the following basic version:
\begin{lemma}\label{lem:poincaresobolev}
  Let $x_0 \in \R^n$ and $r>0$. Then for any $p\geq 1$ and any $1\leq s \leq \frac{n}{n-1}$ there exists a constant $C=C(n,M,p,s)>0$ such that 
  \begin{equation}\label{eq:Vfunction_poincare}
    \left(\dashint_{\ball_r(x_0)} \left(V_p\left(\frac{u -(u)_{\ball_r(x_0)}}{r}\right)\right)^{s} \dif x\right)^{\frac{1}{s}} \leq C \dashint_{\ball_r(x_0)} V_p(\nabla u) \dif x
  \end{equation} 
  holds for all $u \in \sobo^{1,p}(\ball_r(x_0);\mathbb R^M).$ 
\end{lemma}
\begin{proof}
  By translating and rescaling we can assume $x_0=0$ and $r=1,$ and that $(u)_{\ball_1(0)}=0.$ Define, for $x\in\ball_{1}(0)$, a measure $\mu_x$ by $\mu_x := |x-\cdot|^{1-n}\mathscr{L}^{n}$. Then we have $c_1 \leq \mu_x (\ball_1 (0))\leq c_2$ for some constants $c_1 , c_2 >0$ independent of $x$. In consequence, by the usual Riesz potential bound with some $c_{0}=c_{0}(n)>0$
  \begin{align*}
  |u(x)| \leq c_0 \int_{\ball_{1}(0)}\frac{|\nabla u(y)|}{|x-y|^{n-1}}\dif y\qquad \text{for $\mathscr{L}^{n}$-almost all}\;x\in \ball_1 (0), 
  \end{align*}
  we deduce by Jensen's inequality and $\widetilde{\mu}_x := \mu_x / \mu_x (\ball_1 (0))$
  \begin{align}\label{eq:just_poincare}\begin{split}
    \int_{\ball_1(0)} V_p\left(u(x)\right) \dif x & \leq \int_{\ball_{1}(0)} V_{p}\left(c_0 \int_{\ball_{1}(0)} \frac{|\nabla u(y)|}{|x-y|^{n-1}}\dif y  \right)\dif x\\ 
    & \leq \int_{\ball_{1}(0)} V_{p}\left(c_0 c_2\int_{\ball_{1}(0)} |\nabla u(y)|\dif \widetilde{\mu}_x (y)  \right)\dif x\\
     & \!\!\!\!\!\!\!\!\!\!\!\!\!\!\stackrel{\text{Lem.  \ref{lem:Vaux}\ref{item:Vprops5}, Jensen}}{\leq} c(n,p) \int_{\ball_1(0)} \int_{\ball_{1}(0)}V_p(\nabla u(y))\dif \mu_{x}(y) \dif x \\
    & \leq c(n,p)\int_{\ball_1(0)} V_p(\nabla u(y))  \dif y, 
    \end{split}
  \end{align}
  the ultimate estimate being a consequence of Fubini's theorem. Now $\nabla \left(V_p(u)\right) = V_p'(u)\nabla u$ lies in $\lebe^1(\ball;\R^{M\times n})$ using Lemma~\ref{lem:Vaux}~\ref{item:Vprops4} and \ref{item:Vprops8}. Using the Gagliardo-Nirenberg-Sobolev inequality in the second step, we then find 
  \begin{align*}
\int_{\ball_1(0)} V_p\left(u\right)^{s} \dif x & \leq c(s)\int_{\ball_1(0)} (V_p\left(u\right)-(V_p (u))_{\ball_1 (0)})^{s} \dif x\\
& + c(s)\Big(\int_{\ball_1 (0)}V_p (u(x))\dif x\Big)^s \\ 
& \!\stackrel{\text{GNS}}{\leq} c(n,s)\Big(\int_{\ball_{1}(0)}  |\nabla (V_p (u))|\dif x \Big)^{s}+ c(s)\Big(\int_{\ball_1 (0)}V_p (u)\dif x\Big)^s \\ 
& \!\stackrel{\eqref{eq:just_poincare}}{\leq} c(n,s)\Big(\int_{\ball_{1}(0)} V'_p (u)|\nabla u|\dif x \Big)^{s}+ c(s)\Big(\int_{\ball_1 (0)}V_p (\nabla u)\dif x\Big)^s \\ &  =: \mathrm{I} + \mathrm{II}. 
  \end{align*}
For term $\mathrm{I}$, we use Lemma~\ref{lem:Vaux}~~\ref{item:Vprops4A} and ~\ref{item:Vprops8} and to estimate
\begin{align*}
    \mathrm{I} & \leq c(n,s)\Big(\int_{\ball_{1}(0)} V_{p'}(V'_p (u))+V_p(|\nabla u|)\dif x \Big)^{s}\\ 
    & \leq c(n,s)\Big(\int_{\ball_{1}(0)} V_{p'}(V'_p (u))\dif x \Big)^s +c(n,s)\Big(\int_{\ball_1 (0)}V_p(\nabla u)\dif x \Big)^{s}\\ 
   &\leq c(n,s,p)\Big(\int_{\ball_{1}(0)} V_p (u)\dif x \Big)^s +c(s)\Big(\int_{\ball_1 (0)}V_p(\nabla u)\dif x \Big)^{s}. 
\end{align*}
Estimating the penultimate term by \eqref{eq:just_poincare} 
and keeping $\mathrm{II}$ as it is, \eqref{eq:Vfunction_poincare} follows. The proof is complete. 
\end{proof}
Fourth, we gather the natural $\lebe^{p}$-estimates for strongly elliptic systems. The following is sketched in \cite[Proposition 2.11]{GK1} using a well-known method of \textsc{Stampacchia} \cite{Stampacchia}; c.f. \cite[Chpt. 10]{Giusti}. 
\begin{lemma}\label{lem:Lpbounds}
Let $\mathcal{A}$ be a bilinear form on $\R^{N\times n}$ that satisfies $\lambda|z|^{2}\leq \langle \mathcal{A}z,z\rangle \leq \Lambda|z|^{2}$ for all rank-one matrices $z\in\R^{N\times n}$ and some $0<\lambda\leq \Lambda<\infty$. Moreover, let $2 \leq p<\infty$ and, for $H\colon\ball_{1}(0)\to\R^{N}$ and  $\vartheta\colon\partial\!\ball_{1}(0)\to\R^{N}$ to be specified below, consider the system 
\begin{align}\label{eq:sys}
\begin{cases} 
-\mathrm{div}(\mathcal{A}\nabla v ) = H & \text{in}\;\ball_{1}(0),\\ 
v= \vartheta &\text{on}\;\partial\!\ball_{1}(0). 
\end{cases}
\end{align}
Then the following hold: 
\begin{enumerate}
\item\label{item:Lpbound2}  For any $H\in\lebe^{p}(\ball_{1}(0);\R^{N})$, there exists a unique (weak) solution 
\begin{align*}
v\in(\sobo_{0}^{1,p}\cap\sobo^{2,p})(\ball_{1}(0);\R^{N})
\end{align*}
of \eqref{eq:sys} with $\vartheta\equiv 0$ and we have $\|v\|_{\sobo^{2,p}(\ball_{1}(0))}\leq c(n,p,\Lambda/\lambda)\|H\|_{\lebe^{p}(\ball_{1}(0))}$.
  \item\label{item:Lpbound1} For any $G\in\sobo^{1,p}(\ball_{1}(0);\R^{N})$, there exists a unique (weak) solution 
  \begin{align*}
  v\in\sobo_{0}^{1,p}(\ball_1(0);\R^{N})
  \end{align*}
  of  \eqref{eq:sys} with $H=-\mathrm{div}(\mathcal{A}\nabla G)$ and $\vartheta\equiv 0$, 
and we have $\|\nabla v\|_{\lebe^{p}(\ball_{1}(0))}\leq c(p,n,\Lambda/\lambda)\| G\|_{\lebe^{p}(\ball_1(0))}$. 
\item\label{item:Lpbound3} For any $\vartheta\in\sobo^{1-1/p,p}(\partial\!\ball_{1}(0);\R^{N})$ there exists a unique (weak) solution $v\in(\hold^{\infty}\cap\sobo_{\vartheta}^{1,p})(\ball_{1}(0);\R^{N})$ of \eqref{eq:sys} with $H\equiv 0$ ,
and we have $\|v\|_{\sobo^{1,p}(\ball_{1}(0))}\leq c(n,p,\Lambda/\lambda)\|\vartheta\|_{\sobo^{1-1/p,p}(\partial\!\ball_{1}(0))}$. 
Moreover, there exists a constant $C(n,p,\Lambda/\lambda)>0$ such that we have for all $0<r<1$ and all $A \in \mathbb R^{N\times n}$
\begin{align}
\sup_{x\in\ball_{r/2}(0)}|\nabla u(x)-A| + r\sup_{x\in\ball_{r/2}(0)}|\nabla^{2}u(x)|\leq c \dashint_{\ball_{r}(0)}|\nabla u - A|\dif y.
\end{align}
\end{enumerate}
\end{lemma}
In the sequel we will use~\ref{item:Lpbound1}, \ref{item:Lpbound2} with a corresponding $V$-functional estimate when $p \geq 2.$ This simply follows by using Lemma~\ref{lem:Vaux}~\ref{item:Vprops2} so for instance in~\ref{item:Lpbound1} we using the cases $2, p \geq 2$ we have
\begin{equation}\label{eq:Lpbound_vfunction}
  \begin{split}
    \int_{\ball_1(0)} V_p(\nabla^2v) \dif x &\simeq \int_{\ball_1(0)} \lvert \nabla^2 v \rvert^2 + \lvert \nabla^2 v \rvert^p \dif x \\
    &\lesssim \int_{\ball_1(0)} \lvert H \rvert^2 + \lvert H \rvert^p \dif x \simeq \int_{\ball_1(0)} V_p(H) \dif x.
  \end{split}
\end{equation}

\subsection{Proof of the theorem} Based on the preceding lemmas, we now give the
\begin{proof}[Proof of the theorem]
For the maps $\widetilde{u},\widetilde{h}$ as in the theorem and any $\varphi\colon\ball_{r}(x_{0})\to\R^{N}$, we rescale to the unit ball and put for $x\in\ball:=\ball_{1}(0)$
\begin{align*}
\widetilde{U}(x):=\frac{1}{r}\widetilde{u}(x_{0}+rx),\;\;\Psi(x):=\frac{1}{r}(\widetilde{u}-\widetilde{h})(x_{0}+rx),\;\;\;\Phi(x):=\frac{1}{r}\varphi(x_{0}+rx). 
\end{align*}
The statement of the theorem will then follow by rescaling to the ball $\ball_{r}(x_{0})$. 
For future reference, we then record that $\widetilde{U}$ satisfies the Euler-Lagrange system 
\begin{align}\label{eq:ELmain}
\int_{\ball}\langle F'_{\nabla a}(\nabla \widetilde{U}),\nabla\Phi\rangle\dif x = 0\qquad\text{for all}\; \Phi\in\hold^{\infty}_{c}(\ball;\R^{N}). 
\end{align}
Moreover, by the growth bound from Lemma~\ref{lem:Faux}~\ref{item:Fprops1} and a density argument, this holds for all $\Phi \in \sobo_0^{1,\max\{p,p'\}}(\ball;\R^N).$
We will now distinguish between the cases $1<p<2$ and $p\geq 2$. 

\emph{Case 1: $1< p<2$.} We define the truncation operator $\mathbb{T}_{p}\colon\R^{N}\to\R^{N}$ via
\begin{align}\label{eq:truncation}
\mathbb{T}_{p}(z):=\begin{cases} 
z&\;\text{if}\;|z|\leq 1,\\
|z|^{p-2}z&\;\text{if}\;|z|>1. 
\end{cases} 
\end{align}
Recalling that $\mathcal{A}:=F''(\nabla a)$, we then  consider the Legendre-Hadamard elliptic system 
\begin{align}\label{eq:auxSys1}
\begin{cases} 
-\mathrm{div}(\mathcal{A}\nabla\Phi)=\mathbb{T}_{p}(\Psi)&\;\text{in}\;\ball,\\
\Phi=0&\;\text{on}\;\partial\!\ball.
\end{cases} 
\end{align}
We note that since $\Psi\in\lebe^{p}(\ball;\R^{N})$, we have $\mathbb{T}_{p}(\Psi)\in\lebe^{p'}(\ball;\R^{N})$ and so Lemma~\ref{lem:Lpbounds}~\ref{item:Lpbound2} implies that there exists a unique solution $\Phi\in(\sobo_{0}^{1,p'}\cap\sobo^{2,p'})(\ball;\R^{N})$ of~\eqref{eq:auxSys1}. Clearly, for any $\delta>0$, $\delta\Phi$ then is the unique solution of the system 
\begin{align}\label{eq:auxSys1a}
\begin{cases} 
-\mathrm{div}(\mathcal{A}\nabla(\delta\Phi))=\delta\mathbb{T}_{p}(\Psi)&\;\text{in}\;\ball,\\
\Phi=0&\;\text{on}\;\partial\!\ball.
\end{cases} 
\end{align}
Moreover since $p' \geq 2,$ by applying Lemma~\ref{lem:Lpbounds}~\ref{item:Lpbound2} with $2,p'$ in place of $p$ and combining with Lemma~\ref{lem:Vaux}~\ref{item:Vprops1} as in \eqref{eq:Lpbound_vfunction}, we have the estimate
\begin{equation}\label{eq:auxSys_modular}
  \int_{\ball} V_{p'}(\delta\nabla^{2}\Phi) \dif x \leq c \int_{\ball} V_{p'}(\delta\mathbb T_p(\Psi)) \dif x.
\end{equation} 
Noting $(\nabla \Phi)_{\ball} = 0$ since $\Phi$ vanishes on $\partial\!\ball,$ we can estimate for any $\delta>0$ \begin{equation}\label{eq:subquad_phi_estimate}
  \begin{split}
    \left(\int_{\ball} V_{p'}(\delta\nabla \Phi)^{\frac{n}{n-1}} \dif x\right)^{\frac{n-1}n} 
    &\stackrel{\text{\eqmakebox[a1][c]{Lem.~\ref{lem:poincaresobolev}}}}{\leq} c \int_{\ball} V_{p'}(\delta\nabla^2\Phi) \dif x \\
    &\stackrel{\text{\eqmakebox[a1][c]{\eqref{eq:auxSys_modular}}}}{\leq} c\int_{\ball} V_{p'}(\delta\mathbb T_p(\Psi)) \dif x, 
  \end{split}
\end{equation} 
where $c$ is independent of $\delta$. Using $\Psi \in \sobo^{1,p}_0(\ball,\R^N)$ as a test function and noting that $\Phi\in(\sobo_{0}^{1,p} \cap \sobo_{0}^{1,p'})(\ball;\R^{N})$ is also admissible for~\eqref{eq:systemMain} and ~\eqref{eq:ELmain} we obtain
\begin{align*}
\int_{\ball}V_{p}(\Psi)\dif x & \stackrel{\text{\eqmakebox[a2][c]{\text{Lem.} \ref{lem:Vaux}\ref{item:Vprops1}}}}{\leq} c \int_{\ball}\langle\mathbb{T}_{p}(\Psi),\Psi\rangle\dif x\\ 
                              & \stackrel{\text{\eqmakebox[a2][c]{\eqref{eq:auxSys1}}}}{=} c\int_{\ball}\langle\mathcal{A}\nabla\Phi,\nabla\Psi\rangle\dif x\\ 
& \stackrel{\text{\eqmakebox[a2][c]{}}}{=} c\int_{\ball}\langle\mathcal{A}\nabla \Psi,\nabla\Phi\rangle\dif x\\ 
& \stackrel{\text{\eqmakebox[a2][c]{\eqref{eq:systemMain},~\eqref{eq:ELmain}}}}{\leq} c\int_{\ball}\langle(F''(\nabla a)\nabla \widetilde{U}-F'_{\nabla a}(\nabla\widetilde{U})),\nabla\Phi\rangle\dif x \\ 
& \stackrel{\text{\eqmakebox[a2][c]{\text{Lem.} \ref{lem:Faux}\ref{item:Fprops2}}}}{\leq} c\int_{\ball}\omega_m(\lvert \nabla\widetilde U\rvert) \lvert \nabla\widetilde U\rvert \lvert\nabla\Phi\rvert\dif x.
\end{align*}
Now for $\delta>0$ to be determined, by Young's inequality (Lemma~\ref{lem:Vaux}~\ref{item:Vprops8}) we have 
\begin{align*}
\int_{\ball} V_{p}(\Psi) \dif x & \leq c \delta \int_{\ball} V_p(\nabla \widetilde U) \dif x + c \delta \int_{\ball} V_{p'}\left( \omega_m(\lvert\nabla \widetilde U\rvert)\, \delta^{-1} \nabla\Phi\right) \dif x \\ & =: c(\mathrm{I}_1 + \mathrm{I}_2).
\end{align*} 
Now by convexity of $V_{p'}$ the second term can be estimated as
\begin{equation}
  \begin{split}
    \mathrm{I}_2 &\stackrel{\text{\eqmakebox[a2][c]{$\omega_m(\cdot)\leq 1$,\;\text{Lem.} \ref{lem:Vaux}\ref{item:Vprops6}}}}{\leq} \delta \int_{\ball} \omega_m(\lvert\nabla \widetilde U\rvert) V_{p'}\left( \delta^{-1} \nabla\Phi \right) \dif x \\
    &\stackrel{\text{\eqmakebox[a2][c]{H\"older}}}{\leq} \| \omega_m(\lvert \nabla\widetilde U\rvert) \|_{\lebe^n(\ball)}\| \delta V_{p'}(\delta^{-1}\nabla\Phi) \|_{\lebe^{\frac{n}{n-1}}(\ball)} \\
    &\stackrel{\text{\eqmakebox[a2][c]{$\omega_m(\cdot)\leq 1$,\,\eqref{eq:subquad_phi_estimate}}}}{\leq} c\left( \int_{\ball} \omega_m(\lvert\nabla\widetilde U\rvert) \dif x \right)^{\frac1n} \int_{\ball} \delta V_{p'}\left( \delta^{-1}\mathbb T_p(\Psi) \right) \dif x \\
    &\stackrel{\text{\eqmakebox[a2][c]{Jensen}}}{\leq} c\Big(\rho\Big( \int_{\ball} V_p(\lvert\nabla\widetilde U\rvert) \dif x \Big)\Big)^{\frac1n} \int_{\ball} \delta V_{p'}\left( \delta^{-1}\mathbb T_p(\Psi) \right) \dif x,
  \end{split}
\end{equation}
  where $\rho = \omega_m \circ V_p^{-1},$ noting it is concave.
  Now we use Lemma~\ref{lem:Vaux}~\ref{item:Vprops5}, \ref{item:Vprops6} distinguishing between the cases $\delta \leq 1$ and $\delta > 1$ to estimate
  \begin{equation}
    \begin{split}
      \int_{\ball} \delta V_{p'}\left( \delta^{-1}\mathbb T_p(\Psi) \right) \dif x
      &\leq c\left( \delta^{1-p'} \mathbbm{1}_{\{\delta < 1\}} + \delta^{-1} \mathbbm{1}_{\{\delta \geq 1\}} \right) \int_{\ball} V_{p'}\left( \mathbb T_p(\Psi) \right) \dif x\\
      &\leq c \left( \delta^{1-p'} \mathbbm{1}_{\{\delta < 1\}} + \delta^{-1} \mathbbm{1}_{\{\delta \geq 1\}} \right) \int_{\ball} V_{p}\left( \Psi \right) \dif x,
    \end{split}
  \end{equation}
  where we used the estimate $V_{p'}(\mathbb T_p(\Psi)) \lesssim_{p} V_p(\Psi),$ which follows from  Lemma~\ref{lem:Vaux}~\ref{item:Vprops1}.
  Incorporating this into the estimate for $\mathrm{I}_2,$ we arrive at
  \begin{equation}
    \begin{split}
    &\int_{\ball} V_p(\Psi) \dif x 
    \leq c\,\delta \int_{\ball} V_p(\nabla\widetilde U) \dif x \\
    &\qquad+ c \left( \delta^{\red{1-p'}} \mathbbm{1}_{\{\delta < 1\}} + \delta^{\red{-1}} \mathbbm{1}_{\{\delta \geq 1\}} \right) \,\Big(\rho\Big( \int_{\ball} V_p(\nabla\widetilde U) \dif x \Big)\Big)^{\frac 1n} \int_{\ball} V_p(\Psi) \dif x\\
    &\qquad=: \delta A + \left( \delta^{1-p'} \mathbbm{1}_{\{\delta < 1\}} + \delta^{-1} \mathbbm{1}_{\{\delta \geq 1\}} \right) B.
  \end{split}
  \end{equation} 
  Now we wish to extremise $\delta.$ If $A > B$ we take $\delta = (B/A)^{\frac1{p'}} < 1,$ and if $A \leq B$ we take $\delta = (B/A)^{\frac12}.$ This gives
  \begin{equation}
    \int_{\ball} V_p(\Psi) \dif x \leq 
    \begin{cases}
      A^{\frac1{p}} B^{\frac{1}{p'}} & \text{ if } A > B, \\
      A^{\frac12} B^{\frac12} & \text{ if } A \leq B,
    \end{cases}
  \end{equation} 
  so noting that $\rho(\cdot) \leq 1$ we arrive at
  \begin{equation}
    \int_{\ball} V_p(\Psi) \dif x \leq c\,\left(\rho\Big( \int_{\ball} V_p(\nabla\widetilde U) \dif x \Big)\right)^{\frac {p-1}{n}}\int_{\ball} V_p(\nabla\widetilde U) \dif x.
  \end{equation} 
  Hence the result follows by taking $H(t) = \rho(t)^{\frac{p-1}{n}} t.$

\emph{Case 2. $p \geq 2$.} In this case, we put slightly different from above $p_*:=\frac{np}{n+p}$ so that $1<p_*<\min\{n,p\}$ and $(p_*)^{*}=p$. Note that $\Psi$ satisfies 
\begin{align}\label{eq:superquadratic_comparision}
\begin{cases} 
-\mathrm{div}\big(\mathcal{A}\nabla\Psi\big) = -\mathrm{div}(\mathcal{A}\nabla\widetilde{U})&\;\text{in}\;\ball,\\ 
\Psi = 0&\;\text{on}\;\partial\!\ball 
\end{cases} 
\end{align}
by definition of $\Psi$. In consequence, Lemma~\ref{lem:Lpbounds}~\ref{item:Lpbound1} yields by Sobolev's inequality
\begin{align}\label{eq:plarger2A}
  \|\Psi\|_{\lebe^{p}(\ball)} \stackrel{(p_*)^{*}=p}{\leq} c \|\nabla\Psi\|_{\lebe^{p_*}(\ball)} \stackrel{\text{Lem.~\ref{lem:Lpbounds}~\ref{item:Lpbound1}}}{\leq} c\|\nabla\widetilde{U}\|_{\lebe^{p_*}(\ball)} \stackrel{\text{H\"older}}{\leq} c\|\nabla\widetilde{U}\|_{\lebe^{q}(\ball)}
\end{align}
for any choice of $q \in (p_*,p).$ Moreover if $p > 2$ we may choose $q > 2,$ so there is $\lambda \in (0,1)$ for which $q=2\lambda + p(1-\lambda).$
Hence
\begin{align}\label{eq:convexityelementary}
t^{q} & = \exp(q\log(t)) \leq \exp(\lambda\log(t^{2})+(1-\lambda)\log(t^{p}))\leq \lambda t^{2}+(1-\lambda)t^{p}
\end{align}
holds for all $t>0$ and thus from~\eqref{eq:plarger2A} we deduce
\begin{align}\label{eq:Vplarger2}
\begin{split}
\int_{\ball}|\Psi|^{p}\dif x & \stackrel{\eqref{eq:plarger2A}}{\leq} c \Big(\int_{\ball}|\nabla\widetilde{U}|^{q}\dif x\Big)^{\frac{p}{q}} \\ & \stackrel{\eqref{eq:convexityelementary}}{\leq} c\Big(\int_{\ball}|\nabla\widetilde{U}|^{2}+|\nabla\widetilde{U}|^{p}\dif x\Big)^{\frac{p}{q}} \stackrel{\text{Lem.~\ref{lem:Vaux}~\ref{item:Vprops2}}}{\leq} c\Big(\int_{\ball}V_{p}(\nabla\widetilde{U})\dif x\Big)^{\frac{p}{q}}. 
\end{split}
\end{align}
Different from case 1, we now let $\Phi$ be the solution of the Legendre-Hadamard elliptic system 
\begin{align}\label{eq:auxsystemB}
\begin{cases} 
-\mathrm{div}\big(\mathcal{A}\nabla\Phi\big) = \Psi&\;\text{in}\;\ball,\\ 
\Phi = 0&\;\text{on}\;\partial\!\ball. 
\end{cases}
\end{align}
We set $p^{\#} := \frac{np}{n-1}$, so that $p^{\#}<p^{*}$.  Then the Poincar\'e-Sobolev inequality (noting $(\nabla\Phi)_{\ball}=0$) and the usual  $\sobo^{2,p}$-estimates imply
  \begin{equation}\label{eq:super_PhiPsi}
    \| \nabla\Phi \|_{\lebe^{p^{\#}}(\ball)} 
    \leq c \|\nabla^2\Phi\|_{\lebe^p(\ball)} 
    \stackrel{\text{Lem.~\ref{lem:Lpbounds}~\ref{item:Lpbound2}}}{\leq} c\|\Psi\|_{\lebe^p(\ball)} \leq c\left( \int_{\ball} V_p(\Psi) \dif x \right)^{\frac1p}.
  \end{equation} 
  We now argue as in the superquadratic case, so we have
  \begin{equation}\label{eq:pgeq2_maincomparision}
    \begin{split}
      \int_{\ball} \lvert \Psi\rvert^2 \dif x
      &\stackrel{\text{\eqmakebox[a2][c]{\eqref{eq:auxsystemB}}}}{=} \int_{\ball} \langle\mathcal A \nabla \Phi, \nabla \Psi \rangle\,\dif x\\
      &\stackrel{\text{\eqmakebox[a2][c]{\eqref{eq:superquadratic_comparision}}}}{=} \int_{\ball} \langle\mathcal A \nabla \widetilde U, \nabla \Phi\rangle \,\dif x\\
      &\stackrel{\text{\eqmakebox[a2][c]{\eqref{eq:auxsystemB}}}}{=} \int_{\ball} \langle(F''(\nabla)\nabla\widetilde U - F_{\nabla a}'(\nabla\widetilde U)), \nabla\Phi\rangle \,\dif x\\
      &\stackrel{\text{\eqmakebox[a2][c]{Lem.~\ref{lem:Faux}~\ref{item:Fprops2}}}}{\leq} c\int_{\ball} \omega_m(\lvert\nabla\widetilde U\rvert)\lvert \nabla\widetilde U\rvert^{p-1} \lvert\nabla\Phi\rvert \,\dif x\\
      &\stackrel{\text{\eqmakebox[a2][c]{H\"older}}}{\leq} c\|\omega_m(\lvert\nabla\widetilde U\rvert)\|_{\lebe^{np}(\ball)} \| \nabla \widetilde U\|_{\lebe^p(\ball)}^{p-1} \|\nabla\Phi\|_{\lebe^{p^{\#}}(\ball)}, 
    \end{split}
  \end{equation}
  noting that $\frac{1}{np}+\frac{1}{p'}+\frac{1}{p^{\sharp}}=1$. We now estimate the three terms separately; applying Jensen to the concave function $\rho = \omega_m \circ V_p^{-1} \leq 1,$ noting $\lvert \cdot\rvert^p \leq V_p$ by Lemma~\ref{lem:Vaux}~\ref{item:Vprops1}, and using \eqref{eq:super_PhiPsi} we estimate
  \begin{align}\label{eq:Psi2_estimate}
  \begin{split}
    \int_{\ball} \lvert \Psi\rvert^2 \dif x
    & \leq c \left(\rho\Big(\int_{\ball} V_p(\nabla\widetilde U) \dif x\Big)\right)^{\frac1{np}}\times \\ & \times\left( \int_{\ball} V_p(\nabla\widetilde U) \dif x \right)^{\frac{p-1}{p}}\left( \int_{\ball} V_p(\Psi) \dif x \right)^{\frac1p}.
    \end{split}
  \end{align} 
  If $p=2$ we have $V_2(z) = \lvert z\rvert^2$ so the result follows with $H(t) = \rho(t)^{\frac1n}t.$
  Else if $p>2,$ combining \eqref{eq:Vplarger2} and \eqref{eq:Psi2_estimate} using Lemma~\ref{lem:Vaux}~\ref{item:Vprops1} we deduce that
  \begin{equation}
    \begin{split}
      \int_{\ball} V_p(\Psi) \dif x
      &\leq c \int_{\ball} \lvert \Psi\rvert^2 + \lvert \Psi\rvert^p \dif x\\
      &\!\!\!\!\!\!\leq c \left(\rho\Big(\int_{\ball} V_p(\nabla\widetilde U) \dif x\Big)\right)^{\frac1{np}} \left( \int_{\ball} V_p(\nabla\widetilde U) \dif x \right)^{\frac{p-1}p} \left( \int_{\ball} V_p(\Psi) \dif x \right)^{\frac1p} \\
      &\!\!\!\!\!\!\quad+ c \left( \int_{\ball} V_p(\nabla\widetilde U) \dif x \right)^{\frac pq}.
    \end{split}
  \end{equation} 
  Now applying Young's inequality to the first term we can absorb the $V_p(\Psi)$ term to the left-hand side giving
  \begin{align*}
    \int_{\ball} V_p(\Psi) \dif x &\leq c\left(\rho\Big(\int_{\ball} V_p(\nabla\widetilde U)\dif x\Big)\right)^{\frac{1}{n(p-1)}}\int_{\ball} V_p(\nabla\widetilde U)\dif x \| &+ \left( \int_{\ball} V_p(\nabla\widetilde U)\dif x \right)^{\frac pq},
\end{align*}
and so the result follows by taking $H(t) = \rho(t)^{\frac{1}{n(p-1)}}t + t^{\frac pq}$ (recall that $p>q$). 
\end{proof} 

\begin{remark}[Dependence of constants]\label{rem:thm_constant}
  A careful analysis of the above proof reveals that the constant $c$ in the main theorem has the dependence
  \begin{equation}
    c = c(n,N,p,L_m / \nu_m) > 0,
  \end{equation} 
  where $L_m, \nu_m > 0$ are given as in Lemma \ref{lem:Faux}.
  This follows by noting that if $\nu_m>0$ is as in Lemma~\ref{lem:Faux}, by considering $\nu_m^{-1}F$ we can assume that $\nu_m=1.$
  In addition, the superlinear function $H$ is seen to take an explicit form depending on $n, p$ and $\omega_m$ only.
\end{remark}
\begin{remark}[Estimates for the gradient]\label{rem:gradientbounds}
  One may also work purely at the level of gradients, and establish an estimate of the form
  \begin{equation}\label{eq:gradient_harmonixapprox}
    \left(\dashint_{\ball} V_p(\nabla \Psi)^{\theta} \dif x\right)^{\frac 1{\theta}} \leq c \widetilde H\left( \dashint_{\ball} V_p(\nabla \widetilde U) \dif x \right).
  \end{equation}
  with $\theta = \frac{n}{n+1} \in (0,1),$ with $\widetilde H$ a superlinear function.
  This is in the spirit of the sharp $\mathcal{A}$-harmonic approximations results of \textsc{Cruz-Uribe \& Diening} \cite{DieningCruzzi}; in particular, it shows  that the problem of regularity is solely one of higher integrability, and oscillation effects are not an issue.

  The argument is similar to that of the main theorem, where we instead choose our test function $\Phi$ to satisfy the auxiliary system
    \begin{equation}\label{eq:gradient_aux_general}
    \begin{cases} 
      -\mathrm{div}\big(\mathcal{A}\nabla\Phi\big) = -\mathrm{div} \,G&\;\text{in}\;\ball,\\ 
      \Phi = 0&\;\text{on}\;\partial\!\ball, 
    \end{cases}
  \end{equation}
  where $G$ is chosen to satisfy
  \begin{equation}
    G \cdot \nabla\Psi = \begin{cases} V_p(\nabla\Psi)^{\frac{n}{n+1}} &\;\text{if}\; p < 2,\\ \lvert \nabla\Psi \rvert^{\frac{n}{n+1}} &\;\text{if}\; p \geq 2. \end{cases}
  \end{equation}
  For our purposes \eqref{eq:AharmonicSharp} is sufficient however, so we will postpone the proof of the case $p\geq 2$ to the appendix, which proceeds along similar lines to above proof.

\end{remark}

\section{Application to partial regularity}\label{sec:app}
\subsection{Partial regularity}\label{sec:eps_reg}
We conclude the paper by outlining how the main theorem can be used to derive the partial regularity of local minimisers of variational integrals subject to~\eqref{eq:pgrowth}.  We state the corresponding result first: 
\begin{theorem}[Partial regularity]\label{thm:PR1}
Let $F\colon\R^{N\times n}\to\R$ satisfy~\eqref{eq:pgrowth}, $\Omega\subset\R^{n}$ be open and bounded and let $u\in\sobo_{\locc}^{1,p}(\Omega;\R^{N})$ be a local minimiser of the variational integral~\eqref{eq:mainfct}. Moreover, denote the \emph{regular set} 
\begin{align*}
\reg_{u}:=\{x\in\Omega\colon\; u\;\text{is of class $\hold^{1,\alpha}$ for all $0<\alpha<1$ in a neighbourhood of $x$}\}
\end{align*}
and hereafter the \emph{singular set}  $\sing_{u}:=\Omega\setminus\Reg_{u}$. Then we have the decomposition  $\sing_u = \sing_{u,1} \cup \sing_{u,2}$ where
\begin{align}\label{eq:charsingset}
\sing_{u,1} & = \left\{x \in \Omega : \liminf_{r \searrow 0}\, \dashint_{\ball_r(x)} V_p(\nabla u - (\nabla u)_{\ball_r(x)}) \dif y > 0 \right\}, \\
\label{eq:charsingset2}
\sing_{u,2} & = \left\{x \in \Omega : \liminf_{r \searrow 0}\,\lvert(\nabla u)_{\ball_r(x)}\rvert = +\infty \right\}.
\end{align}
Especially, $\reg_{u}$ is relatively open in $\Omega$,  $\mathscr{L}^{n}(\sing_{u})=0$ and hence $u$ is of class $\hold^{1,\alpha}$ in $\reg_{u}$ for any $\alpha\in(0,1)$.
\end{theorem}
The previous theorem is essentially due to \textsc{Evans} \cite{Evans} and \textsc{Acerbi \& Fusco} \cite{AcerbiFusco} for the superquadratic regime $p\geq 2$ and due to \textsc{Carozza, Fusco \& Mingione} in the subquadratic case $1<p<2$.
Theorem \ref{thm:PR1} arises as an $\varepsilon$-regularity result, in turn being approached by an excess decay and the Campanato-Meyers characterisation of H\"{o}lder continuity. Since the passage from the excess decay to H\"{o}lder continuity is by now standard, we will focus on the former. 

Besides our main theorem, the present proof of Theorem \ref{thm:PR1} only requires the Caccioppoli inequality of the second kind: Subject to the hypotheses from~\eqref{eq:pgrowth} and given $m>0$, there exists a constant $c=c(n,N,p,L_m/\nu_m)>0$  such that 
\begin{align}\label{eq:Cacc2}
\dashint_{\ball_{r}(x_{0})}V_{p}(\nabla (u-a))\dif x \leq c \dashint_{\ball_{2r}(x_{0})}V_{p}\Big(\frac{u-a}{r}\Big)\dif x 
\end{align}
holds whenever $u\in\sobo_{\locc}^{1,p}(\Omega;\R^{N})$ is a local minimiser of $\mathscr{F}$,  $a\colon\R^{n}\to\R^{N}$ is affine-linear with $|\nabla a|\leq m$, and $\ball_{2r}(x_{0})\Subset\Omega$. To this end, we define the \emph{excess} of $u$ on the ball $\ball_R (x_0)\Subset\Omega$ via 
$$\mathbb{E}_{p}[u;\ball_{R}(x_{0})]:=\dashint_{\ball_R (x_0)}V_{p}(\nabla u-(\nabla u)_{\ball_{R}(x_{0})})\dif x$$ 
and assume that we have for some $M>0$ and $0<\varepsilon<1$
\begin{align}\label{eq:smallnessconds} \begin{split}
&\left\vert \dashint_{\ball_R (x_0)} \nabla u\dif x \right\vert <M,\\
&\dashint_{\ball_{R}(x_{0})}|\nabla u - (\nabla u)_{\ball_{R}(x_{0})}|^p \dif x < \varepsilon.
\end{split}\end{align}
Without loss of generality, we assume $0<R<1$ in the sequel. 
We let $a\colon\R^{n}\to\R^{N}$ be given by $a(x):=(\nabla u)_{\ball_{R}(x_{0})}(x-x_{0})+(u)_{\ball_{r}(x_{0})}$, so that we  have $|\nabla a|\leq M$ by $\eqref{eq:smallnessconds}_{1}$, and write $\widetilde{u} = u -a$ as before. In the following, let $h$ be the unique solution of the Legendre-Hadamard elliptic system 
\begin{align}\label{eq:hdef}
\begin{cases}
-\mathrm{div}(F''((\nabla u)_{\ball_{R}(x_{0})})\nabla h)= 0& \;\text{in}\;\ball_R (x_0), \\ 
h= u&\;\text{on}\;\partial\!\ball_R (x_0).
\end{cases}
\end{align}
In particular, $\widetilde{h}:=h-a$ satisfies 
\begin{align*}
-\mathrm{div}(F''((\nabla u)_{\ball_{R}(x_{0})})\nabla \widetilde{h})= 0\qquad\text{in}\;\ball_R (x_0),
\end{align*}
with its own boundary values along $\partial\!\ball_{R}(0)$. 
As a consequence, we obtain by Lemma \ref{lem:Lpbounds} \ref{item:Lpbound3} for any $1<r\leq p$
\begin{align}\label{eq:hbound0}\begin{split} 
\Big(\dashint_{\ball_{R}(x_{0})}|\nabla\widetilde{h}|\dif x \Big) & \leq  \Big(\dashint_{\ball_{R}(x_{0})}|\nabla\widetilde{h}|^{r}\dif x \Big)^{\frac{1}{r}} \\ 
& \leq c\Big(\dashint_{\ball_{R}(x_{0})}\left\vert\frac{\widetilde{u}}{R}\right\vert^{r} \dif x + \dashint_{\ball_{R}(x_{0})}|\nabla\widetilde{u}|^{r}\dif x \Big)^{\frac{1}{r}} \\ & \leq c\Big(\dashint_{\ball_{R}(x_{0})}|\nabla\widetilde{u}|^{r}\dif x\Big)^{\frac{1}{r}} (\stackrel{\eqref{eq:smallnessconds}_{2}}{<} c\varepsilon^\frac{1}{p}<c)
\end{split}
\end{align}
with $c = c(n,N,p,r,L_m/\nu_m)$, where we used Poincar\'{e}'s inequality in the last line noting that $(\widetilde{u})_{\ball_{R}(x_{0})}=0$.
Again by Lemma \ref{lem:Lpbounds} \ref{item:Lpbound3}, this implies
\begin{align}\label{eq:hbound1}\begin{split}
|\nabla \widetilde{h}(x_{0})|  \leq c\dashint_{\ball_{R/2}(x_{0})}|\nabla \widetilde{h}|\dif y  \leq c\Big(\dashint_{\ball_{\ball_{R}(x_{0})}}|\nabla \widetilde{h}|^{p}\dif y\Big)^\frac{1}{p} \stackrel{\eqref{eq:hbound0}}{<}c\varepsilon^\frac{1}{p} 
\end{split}
\end{align}
with a constant $c=c(n,N,p,L_{m}/\nu_{m})>0$, and thus 
\begin{align}\label{eq:hbound2}
|\nabla h(x_{0})| \leq |\nabla \widetilde{h}(x_{0})| + |\nabla a| < c + M =: m, 
\end{align}
fixing the threshold number $m>0$ in the sequel. 

Now let $0<\sigma<\frac{1}{3}$ be arbitrary but fixed. By Jensen's inequality we first deduce for some constant $c=c(p)>0$
\begin{align}\label{eq:jensenoptimal}
\dashint_{\ball_{\sigma R}(x_{0})}V_p (\nabla u - (\nabla u)_{\ball_{R}(x_{0})})\dif x \leq c \inf_{\pi\,\text{affine-linear}}\dashint_{\ball_{\sigma R}(x_{0})}V_p (\nabla (u-\pi))\dif x.
\end{align}
We now set $b(x):=\nabla h(x_{0}) (x-x_{0}) + h(x_{0})$, so that $|\nabla b|<m$ by \eqref{eq:hbound2}.  
Recalling our notation $\widetilde{u}:=u-a$, we then invoke the Caccioppoli inequality~\eqref{eq:Cacc2} with $m>0$ as in ~\eqref{eq:hbound2} to find for $0<\sigma<\frac{1}{3}$
\begin{equation}\begin{split}\label{eq:excess_bound1}
\mathbb{E}_{p}[u;\ball_{\sigma R}(x_{0})] & \stackrel{\eqref{eq:Cacc2},\,\eqref{eq:jensenoptimal}}{\leq}c\dashint_{\ball_{2\sigma R}(x_{0})} V_{p}\Big(\frac{u-b}{\sigma R} \Big)\dif x\\ & \;\;\;\; =  c\dashint_{\ball_{2\sigma R}(x_{0})}V_{p}\Big(\frac{(\widetilde{u}-\widetilde{h})-(b-h)}{\sigma R} \Big)\dif x\\
& \!\!\!\stackrel{\text{Lem.}~\ref{lem:Vaux}~\ref{item:Vprops8}}{\leq} c\dashint_{\ball_{2\sigma R}(x_{0})}V_{p}\Big(\frac{\widetilde{u}-\widetilde{h}}{\sigma R}\Big)\dif x +c\dashint_{\ball_{2\sigma R}(x_{0})}V_{p}\Big(\frac{b-h}{\sigma R} \Big)\dif x \\
&\!\!\!=: \mathrm{I} + \mathrm{II}.
\end{split}\end{equation}
The estimation of $\mathrm{I}$ is carried out by virtue of our main theorem, whereas $\mathrm{II}$ is dealt with by routine bounds on $\mathcal{A}$-harmonic maps. More precisely, recalling the function $H$ from the main theorem, we have 
\begin{equation}\begin{split}\label{eq:excess_I1bound}
\mathrm{I} & \stackrel{\text{Lem.}~\ref{lem:Vaux}~\ref{item:Vprops5}}{\leq} \frac{c}{\sigma^{n+\max\{2,p\}}} \dashint_{\ball_{R}(x_{0})}V_{p}\Big(\frac{\widetilde{u}-\widetilde{h}}{R}\Big)\dif x \\ & \;\;\;\;\;\leq \frac{c}{\sigma^{n+\max\{2,p\}}} H\Big(\dashint_{\ball_{R}(x_{0})}V_{p}(\nabla\widetilde{u})\dif x \Big),
\end{split}\end{equation}
with $c = c(n,N,p,L_m/\nu_m)$ by Remark \ref{rem:thm_constant}. We turn to the estimation of $\mathrm{II}$. For future reference, we fix $r:=\min\{2,p\}$, so that the function $\mathfrak{V}_{p}(t):=V_{p}(t^{\frac{1}{r}})$ is convex by Lemma \ref{lem:Vaux} \ref{item:Vprops4}. We first note that we have for any $x\in\ball_{2\sigma R}(x_{0})$
\begin{align}\label{eq:harmonicbounds}
\begin{split}
\sup_{x\in\ball_{2\sigma R}(x_{0})}\frac{|b(x)-h(x)|}{\sigma R} & \leq c\sup_{x\in\ball_{2\sigma R}(x_{0})}\frac{|x-x_{0}|^{2}}{\sigma R}|\nabla^{2}\widetilde{h}(x)| \\ 
& \!\!\!\!\!\!\!\!\stackrel{\text{Lem. \ref{lem:Lpbounds}\ref{item:Lpbound3}}}{\leq} c\sigma \dashint_{\ball_{R}(x_{0})}|\nabla \widetilde{h}|\dif x \\
& \leq c\sigma \Big(\dashint_{\ball_{R}(x_{0})}|\nabla\widetilde{h}|^{p}\dif x\Big)^{\frac{1}{p}} \leq \mathtt{C},
\end{split}
\end{align}
the ultimate estimate being valid by $\eqref{eq:hbound1}_{2},\eqref{eq:hbound1}_{3}$, where $\mathtt{C}=\mathtt{C}(\green{n},N,p,L_{m}/\nu_{m})>0$ is a constant. Therefore, Lemma \ref{lem:Vaux} \ref{item:Vprops2a} yields for another constant $c=c(\green{n},N,p,m,L_{m}/\nu_{m})>0$ that
\begin{align}\label{eq:excess_I2bound} 
\begin{split}
\mathrm{II} & \leq c\sup_{x\in\ball_{2\sigma R}(x_{0})}\left\vert \frac{b(x)-h(x)}{\sigma R}\right\vert^{2} \stackrel{\eqref{eq:harmonicbounds}_{1},\,\eqref{eq:harmonicbounds}_{2}}{\leq} c\sigma^{2}\Big(\dashint_{\ball_{R}(x_{0})}|\nabla \widetilde{h}(x)|\dif x\Big)^{2} \\ 
& \!\!\!\!\!\!\!\!\!\!\!\!\!\stackrel{\eqref{eq:hbound0},\,\text{Lem. \ref{lem:Lpbounds}\ref{item:Lpbound3}}}{\leq} c\sigma^{2} V_{p}\Big(\dashint_{\ball_{R}(x_{0})}|\nabla \widetilde{h}|\dif x \Big) \\ 
& \!\!\!\!\stackrel{\text{Jensen}}{\leq}  c\sigma^{2}\mathfrak{V}_{p}\Big(\dashint_{\ball_{R}(x_{0})}|\nabla \widetilde{h}|^{r}\dif x \Big)
 \stackrel{\eqref{eq:hbound0}}{\leq}  c\sigma^{2}\mathfrak{V}_{p}\Big(\dashint_{\ball_{R}(x_{0})}|\nabla \widetilde{u}|^{r}\dif x \Big)
\\ & \!\!\!\!\stackrel{\text{Jensen}}{\leq}  c\sigma^{2}\dashint_{\ball_{R}(x_{0})}\mathfrak{V}_{p}(|\nabla\widetilde{u}|^{r})\dif x = c\sigma^{2}\dashint_{\ball_{R}(x_{0})}V_{p}(\nabla\widetilde{u})\dif x
\end{split}
\end{align}
by definition of $\mathfrak{V}_{p}$. Thus, combining the estimates for $\mathrm{I}$ and $\mathrm{II}$ yields
\begin{align*}
\mathrm{I}+\mathrm{II} \leq c\Big(\frac{1}{\sigma^{n+\max\{2,p\}}} H\Big(\dashint_{\ball_{R}(x_{0})}V_{p}(\nabla\widetilde{u})\dif x \Big)+ \sigma^{2}\dashint_{\ball_{R}(x_{0})}V_{p}(\nabla \widetilde{u})\dif x\Big), 
\end{align*}
whereby we conclude with the auxiliary function $\Psi(t):=H(t)/t$ that
\begin{align}\label{eq:pre-iterate}
\mathbb{E}_{p}[u;\ball_{\sigma R}(x_{0})] \leq c\Big(\frac{1}{\sigma^{n+\max\{2,p\}}}\Psi(\mathbb{E}_{p}[u;\ball_{R}(x_{0})]) + \sigma^{2} \Big)\mathbb{E}_{p}[u;\ball_{R}(x_{0})].
\end{align}
Estimate~\eqref{eq:pre-iterate} now gives rise to the following excess decay estimate: 
\begin{corollary}
Let $F\colon \R^{N\times n}\to \R$ be a variational integrand that satisfies~\eqref{eq:pgrowth} and let $\omega_{m}\colon[0,\infty]\to [0,1]$ be the modulus of continuity from Lemma~\ref{lem:Faux}. Given $0<\alpha<1$ and $m>0$, there exist parameters  $0<\sigma=\sigma(n,N,p,\alpha,\ell,L,m,\omega_{m})<1$ and $0<\varepsilon=\varepsilon(n,N,p,\alpha,\ell,L,m,\omega_{m})<1$ such that every local minimiser $u\in\sobo_{\locc}^{1,p}(\Omega;\R^{N})$ satisfies the following: If $x_{0}\in\Omega$ and $0<R_{0}<1$ are such that  $\ball_{r}(x_{0})\Subset\Omega$ and  
\begin{align}
\mathbb{E}_p [u;\ball_{R_{0}}(x_{0})]<\varepsilon\;\;\;\text{and}\;\;\;|(\nabla u)_{\ball_{R_{0}}(x_{0})}|<m
\end{align}
then we have 
\begin{align}\label{eq:decay1}
\mathbb{E}_p [u;\ball_{\sigma R_{0}}(x_{0})] \leq \sigma^{1+\alpha} \mathbb{E}_p [u;\ball_{R_{0}}(x_{0})].
\end{align}
\end{corollary}
\begin{proof}
Let $\alpha\in(0,1)$ be given. We choose $\varepsilon_0 >0$ so small such that $\mathbb{E}[u;\ball_{R_{0}}(x_{0})]<\varepsilon_0$ and $|(\nabla u)_{\ball_{R_{0}}(x_{0})}|<m$ imply the validity of \eqref{eq:pre-iterate} for any $\sigma\in(0,\frac{1}{3})$. Denoting $c>0$ the constant from $\eqref{eq:pre-iterate}$, we choose $\sigma\in(0,\frac{1}{3})$ such that $2c\sigma^2 \leq \sigma^{1+\alpha}$. Since $\Psi$ satisfies $\limsup_{t\searrow 0}\Psi(t)=0$, we find $\varepsilon_1 >0$ such that $0<t<\varepsilon_1$ implies $\Psi(t)\leq \sigma^{n+2+\max\{2,p\}}$. We subsequently choose $
0<\varepsilon < \min\left\{\varepsilon_0, \varepsilon_1 \right\}$, whereby we infer~\eqref{eq:decay1}  from~\eqref{eq:pre-iterate} at once. The proof is complete. 
\end{proof}
Working from the previous corollary, the proof of Theorem~\ref{thm:PR1} is routine, cf. \cite{Beck,G1,GK1}.

\begin{remark}[On condition~\eqref{eq:vanishatzero}]\label{rem:important}
As the above proof shows, the requisite excess decay estimate requires a superlinear decay of $\Psi$ in \eqref{eq:pre-iterate}, as we cannot compensate the blow-up of $\frac{1}{\sigma^{n+2}}$ by a suitable smallness assumption. This superlinear power decay, in turn being manifested by~\eqref{eq:vanishatzero}, is everything we need to make the excess decay work. 
\end{remark} 
\subsection{Uniformly controlled growth}\label{sec:uniformlycontrolledgrowth}
In this section we will assume $p \geq 2$, the subquadratic case being slightly more technical but less natural to have a power-type decay as shall be imposed in the following. More precisely, we will show that the results of Section \ref{sec:eps_reg} can be strengthened subject to the additional hypothesis
\begin{equation}\label{eq:ucontrolled_growth}
  z\mapsto\frac{F''(z)}{(1+\lvert z\rvert)^{p-2}}\quad \text{ is bounded and uniformly continuous on $\R^{Nn}$}, 
\end{equation}
allowing us to discard the part $\sing_{u,2}$ in the description of the singular set~\eqref{eq:charsingset}. Condition~\eqref{eq:ucontrolled_growth} is satisfied for instance if $F \in \hold^3 (\R^{N\times n})$ and
\begin{equation}
  \lvert F'''(z)\rvert \leq L(1 + \lvert z\rvert)^{p-3}\qquad\text{for all}\;z\in\R^{N\times n}.
\end{equation} 
In particular this holds if $F$ is a polynomial, and hence includes the example of \textsc{Alibert-Dacorogna-Marcellini} \cite{AlibertDacorogna,DacorognaMarcellini}. In the following, we say that \eqref{eq:pgrowth} \emph{holds with uniform $\ell(>0)$} provided the constant $\ell_{m}>0$ as in $\eqref{eq:pgrowth}_{2}$ can be chosen independently of $m>0$. 
\begin{theorem}\label{thm:uniform_epsreg}
  Let $F$ satisfy \eqref{eq:pgrowth} with uniform $\ell$ and \eqref{eq:ucontrolled_growth} with $p \geq 2.$
  Then for each $\alpha \in (0,1),$ there is $\sg>0$ such that if $\Omega \subset \R^n$ is open, $u \in \sobo^{1,p}_{\locc}(\Omega;\R^N)$ is a local minimiser of \eqref{eq:mainfct}, and $\ball_R(x_0) \Subset \Omega$ is a ball for which we have
  \begin{equation}
    \mathbb{E}_p[u;\ball_R(x_0)] < \sg,
  \end{equation} 
  then $u$ is of class $\hold^{1,\alpha}$ in $\ball_{\frac R2}(x_0).$
\end{theorem}
We emphasise in particular that $\sg$ is \emph{independent of $\lvert (\nabla u )_{\ball_R(x_0)}\rvert$}, and does not require any normalisation.
As an application, we can adapt the results of \textsc{Mingione} and the last named author  \cite{KristensenMingione} to deduce the following singular set estimate: 
\begin{corollary}\label{cor:uniform_singset}
  Let $F$ satisfy \eqref{eq:pgrowth} \textcolor{blue}{}with uniform $\ell$ and \eqref{eq:ucontrolled_growth} with $p \geq 2.$
  Then given $m>0,$ there is $\delta = \delta(n,N,p,F,m)>0$ such that the following holds: 
  Suppose $u \in \sobo^{1,p}_{\locc}(\Omega;\R^N)$ is a local minimiser of \eqref{eq:mainfct}, and we additionally have $\nabla u \in \mathrm{BMO}(\Omega;\R^{N\times n})$. 
  Then the singular set $\sing_u$ as defined in \eqref{eq:charsingset} is given by $\sing_{u}=\sing_{u,1}$ and satisfies the \emph{Hausdorff dimension estimate}
  \begin{equation}
    \dim_{\mathscr H}(\sing_{u}) < n-\delta.
  \end{equation} 
\end{corollary}

\begin{proof}[Proof of Theorem {\ref{thm:uniform_epsreg}}]
Given an affine-linear map $a\colon\R^{n}\to \R^{N}$, we recall from~\eqref{eq:linearisedintegrands} the linearised integrand $F_{\nabla a}$, which can be written as
\begin{equation}\label{eq:linearisedintegrands2}
  F_{\nabla a}(z) = \int_0^1 (1-t)F''(\nabla a + t z) [z,z] \dif x,\qquad z\in\R^{N\times n}.
\end{equation}
We then record the following estimates on $F_{\nabla a}$: For all $z\in\R^{N\times n}$ we have 
  \begin{align}
    \lvert F_{\nabla a}(z)\vert & \leq \widetilde L (1 + \lvert \nabla a\rvert^2 + \lvert z\rvert^2)^{\frac{p-2}2}\lvert z\rvert^2,\label{eq:uniform_shifted_growth}\\
    \lvert F_{\nabla a}'(z)\vert & \leq \widetilde L (1 + \lvert \nabla a\rvert^2 + \lvert z\rvert^2)^{\frac{p-2}2}\lvert z\rvert,\label{eq:uniform_shifted_growth2}\\
    \lvert F_{\nabla a}'(z) - F''(\nabla a) z \rvert & \leq \widetilde L (1 +\lvert \nabla a\rvert^2 + \lvert z\rvert^2)^{\frac{p-2}2} \omega(\lvert z\rvert) \lvert z\rvert,\label{eq:uniform_shifted_cts}
  \end{align}
  where $\omega \colon [0,\infty) \to [0,1]$ is a non-decreasing concave function such that $\lim_{t \searrow 0} \omega(t) = 0$.
  Here $\widetilde{L}>0$ is a constant only depending on $n, N, p$ and the upper bound for  \eqref{eq:ucontrolled_growth}, which we will denote by $L_1$, and $\omega$ depends only on the modulus of continuity of \eqref{eq:ucontrolled_growth} which we denote by $\widetilde\omega.$
  Then \eqref{eq:uniform_shifted_growth} follows from the boundedness assumption and \eqref{eq:linearisedintegrands2}, whereas \eqref{eq:uniform_shifted_growth2} is established by estimating similarly as in \eqref{eq:firstderivativebound} and employing \eqref{eq:ucontrolled_growth}. 
For \eqref{eq:uniform_shifted_cts} we will use the pointwise estimate
\begin{align}\label{eq:EricJohnson}
\begin{split}
     \lvert F''(z) - F''(w) \rvert &\leq (1+\lvert z \rvert)^{p-2} \left\lvert\frac{F''(z)}{(1+\lvert z\rvert)^{p-2}} - \frac{F''(w)}{(1+\lvert w\rvert)^{p-2}} \right\rvert \\
      &+ \frac{\lvert F''(w) \rvert}{(1+\lvert w \rvert)^{p-2}}  \underbrace{|(1+|z|)^{p-2}-(1+|w|)^{p-2}|}_{=:G(z,w)}. 
      \end{split}
\end{align}
Evidently $G(z,w)\leq 2(1+|z|+|w|)^{p-2}$, and assuming that $p>2$ and that $|z|\neq |w|$ we can also estimate
\begin{align*}
G(z,w) & \leq \int_{1+|w|}^{1+|z|} \frac{\dif}{\dif s} s^{p-2}\dif s \leq (p-2) (1+|z|+|w|)^{p-2}|z-w|,
\end{align*}
distinguishing between the cases $2 < p < 3$ and $p \geq 3.$ 
Hence \eqref{eq:EricJohnson} implies, upon enlarging $\widetilde{L}$ and setting $\omega(t) = \widetilde\omega(t) + \min\{1,t\}$, that 
\begin{align}\label{eq:intheairtonight} 
\lvert F''(z) - F''(w) \rvert & \leq \widetilde{L}\omega(|z-w|)(1+|z|+|w|)^{p-2}\qquad\text{for all}\;z,w\in\R^{N\times n}. 
\end{align}
Finally, to arrive at \eqref{eq:uniform_shifted_cts}, we start as in \eqref{eq:joestriani} and employ \eqref{eq:intheairtonight}. 
  Moreover by Lemma~\ref{lem:Faux}~\ref{item:Fprops3} we have the Legendre-Hadamard condition
  \begin{equation}\label{eq:uniform_shifted_lh}
    F''(\nabla a)[\xi\otimes\eta,\xi\otimes\eta] \geq \ell (1+\lvert z_0\rvert^2)^{\frac{p-2}2} \lvert \xi\rvert^2\lvert\eta\rvert^2
  \end{equation} 
  for all $\xi \in \R^N,$ $\eta \in \R^n,$
  and directly from $\eqref{eq:pgrowth}_2$ we deduce that
  \begin{equation}\label{eq:uniform_shifted_qc}
    \fint_{\ball} F_{\nabla a}(\nabla \varphi)  \dif x \geq c_p\ell \fint_{\ball} (1+ \lvert \nabla a\rvert +\lvert \nabla \varphi\rvert )^{\frac{p-2}2} \lvert \nabla \varphi\rvert^2 \dif x.
  \end{equation} 
  Using \eqref{eq:uniform_shifted_growth} and \eqref{eq:uniform_shifted_qc} we can establish the Caccioppoli inequality of the second kind, namely \eqref{eq:Cacc2}, in the usual way.
  However the associated constant $C$ now depends only on $n, N, p, \widetilde L, \ell;$ in particular it is independent of our choice of $\nabla a.$
  Similarly if $L_m, \nu_m$ are as in Lemma~\ref{lem:Faux}, we can ensure the quotient $L_m/\nu_m$ depends on $p, \widetilde L, \ell$ only.
  Hence by Remark~\ref{rem:thm_constant}, we have the constants in the $\mathcal{A}$-harmonic approximation theorem can also be chosen to be independent of $m.$
  Now we can carry out the excess decay argument as below; here we no longer need to estimate the affine maps $a(x),$ $b(x)$ so letting $h$ as in \eqref{eq:hdef} we can estimate for $\sigma \in (0,\frac13)$
  \begin{equation}
    \mathbb{E}_p[u;\ball_{\sigma R}(x_0)] \leq \frac{c}{\sigma^{n+\max\{2,p\}}} H\left( \mathbb{E}_p[u;\ball_{R}(x_0)]\right) + c \dashint_{\ball_{2\sigma R}(x_0)} V_p\left(\frac{b-h}{\sigma R}\right) \dif x,
  \end{equation}
  using \eqref{eq:excess_bound1}, \eqref{eq:excess_I1bound}, where we use the Caccioppli inequality and $\mathcal{A}$-harmonic approximation results.
  For the second term a slight modification is necessary; \eqref{eq:harmonicbounds} continues to hold true, and we will argue similarly as in \eqref{eq:excess_I2bound}.
  The only difference is that we will not use the equivalence $V_p(z) \simeq \lvert z\vert^2$ for $\lvert z \rvert \lesssim m,$ and instead use Lemma~\ref{lem:Vaux}~\ref{item:Vprops6} to pull out the factor $\sigma.$
  That is, for $1 < r < \min\{2,p\}$ we estimate
  \begin{equation}
    \begin{split}
      \dashint_{B_{2\sigma R}(x_0)} V_p\left(\frac{b-h}{\sigma R}\right) \dif x
      &\leq V_p\left( \sup_{x \in \ball_{2\sigma R}(x_0)} \left\lvert \frac{b(x)-h(x)}{\sigma R}\right\rvert \right) \\
      &\leq c V_p\left( \sigma\left( \dashint_{\ball B_R(x_0)} \lvert \nabla \tilde h\rvert^r \dif x\right)^{\frac1r}\right)\\
      &\leq c \sigma^{\min\{2,p\}} V_p\left( \left(\dashint_{\ball_R(x_0)} \lvert\nabla\tilde u\rvert^r \dif x \right)^{\frac1r}\right) \\
      &\leq c \sigma^{\min\{2,p\}} \mathbb{E}_p[u;\ball_R(x_0)],
    \end{split}
  \end{equation}
  from which we can deduce the slightly modified excess decay estimate
  \begin{align}
    \mathbb{E}_{p}[u;\ball_{\sigma R}(x_{0})] \leq c\Big(\frac{1}{\sigma^{n+\max\{2,p\}}}\Psi(\mathbb{E}_{p}[u;\ball_{R}(x_{0})]) + \sigma^{\min\{2,p\}} \Big)\mathbb{E}_{p}[u;\ball_{R}(x_{0})],
  \end{align}
  where $c=c(n,N,p,\widetilde L,\ell)>0,$ and $\Psi$ depends on $n,p$ and $\omega$ only.
  Now iterating this gives the desired conclusion.
\end{proof}
\begin{proof}[Proof of Corollary {\ref{cor:uniform_singset}}]
This will follow from the Carleson condition
\begin{equation}\label{eq:carleson_condition}
  I:= \fint_{\ball_R(x_0)} \int_0^R \fint_{\ball_R(x_0)} \mathbb E_p [u;\ball_{r}(x)] \, \frac{\dif r}r \dif x \leq C_0,
\end{equation} 
from which one can argue exactly as in \cite[Theorem 5.1]{KristensenMingione} to show $\Omega' \cap \Sigma_u$ is $(\lambda,\kappa)$-porous for all $\Omega' \Subset \Omega,$ with $\lambda = \lambda(p,C_0,\sg)$ where $C_0 > 0$ is as above and $\sg>0$ is as in Theorem \ref{thm:uniform_epsreg}.

To show \eqref{eq:carleson_condition}, we choose
\begin{equation}
    a(x) = (u)_{\ball_{2R}(x_0)} + (\nabla u)_{\ball_{2R}(x_0)} \cdot (x-x_0)
\end{equation}
in \eqref{eq:Cacc2}, which together with~Lemma~\ref{lem:Vaux}~\ref{item:Vprops1} gives
\begin{equation}
  I \leq C \fint_{\ball_R(x_0)} \int_0^R \fint_{\ball_R(x_0)}\left( \lvert u - a \rvert^2 + \lvert u - a \rvert^p \right)\, \frac{\dif r}r \dif x.
\end{equation}
Now we apply the estimate of \textsc{Dorronsoro} proved in \cite[Lemma 2.2]{KristensenMingione} to arrive at
\begin{equation}
\begin{split}
  I &\leq C \fint_{\ball_4R(x_0)} \left(\lvert \nabla u - (\nabla u)_{\ball_{4R}(x_0)}\rvert^2 + \lvert \nabla u - (\nabla u)_{\ball_{4R}(x_0)}\rvert^p\right) \dif x \\
    &\leq C \left([\nabla u]_{\mathrm{BMO}(\ball_{4R}(x_0))}^2 + [\nabla u]_{\mathrm{BMO}(\ball_{4R}(x_0))}^p \right),
\end{split}
\end{equation}
where we have used the John-Nirenberg inequality (see for instance \cite[Theroem 2.11]{Giusti}) in the last line.
We remark that in \cite{KristensenMingione} the estimate of Dorronsoro was only proven in the case $p=2,$ however the corresponding estimate in $\lebe^p$ for $p \geq 2$ follows by applying the Mikhlin multiplier theorem in place of the Plancherel theorem.
This establishes \eqref{eq:carleson_condition}, from which the result follows.
\end{proof}

{\small
\textbf{Acknowledgment.} F.G. gratefully acknowledges fincancial support through the Hector Foundation, Project No. FP 626/21. C.I. was supported by the EPSRC [EP/L015811/1].
}
\section*{Appendix}
For completeness, we use this appendix to give the quick argument for the $\lebe^{p}$-bounds displayed in Lemma \ref{lem:Lpbounds}, and an outline of the claim in Remark \ref{rem:gradientbounds} when $p \geq 2$.
\begin{proof}[Sketch of proof of Lemma \ref{lem:Lpbounds}]
For~\ref{item:Lpbound2}, \ref{item:Lpbound1}, the case $2 \leq p < \infty$ is shown in \cite[Chapter 10]{Giusti} using the method of Stampacchia.
In both cases we can extend this to the range $1 < p < 2$ by means of a duality argument.
Here it suffices to establish an a-priori estimate, for which we assume the right-hand side $G$ and $H$ respectively lie in $L^2,$ and let $v,w$ be the respective weak solutions.
For~\ref{item:Lpbound1} we choose $\varphi$ to be the weak solution to the dual problem
\begin{align*}
\begin{cases} 
-\mathrm{div}(\mathcal{A}^*\nabla \varphi ) = -\mathrm{div}\left( \lvert\nabla v\rvert^{p-2}\nabla v\right) & \text{in}\;\ball_{1}(0),\\ 
\varphi= 0&\text{on}\;\partial\!\ball_{1}(0), 
\end{cases} 
\end{align*}
where $\mathcal{A}^*[\xi,\zeta] = \mathcal{A}[\zeta,\xi]$ for all $\xi,\zeta \in \mathbb R^{Nn}.$
Then testing the equation for $v$ with $\varphi$ and applying~\ref{item:Lpbound1} with $p'\geq 2$ we deduce that
\begin{equation}
  \begin{split} 
    &\int_{\ball_1(0)} \lvert \nabla v\rvert^p \dif x = \int_{\ball_1(0)} \langle \mathcal{A}\nabla v, \nabla \varphi \rangle \dif x = \int_{\ball_1(0)} \langle \nabla\varphi, H \rangle \dif x \\
    &\quad\lesssim \lVert \nabla\varphi \rVert_{\lebe^{p'}(\ball_1(0)}\lVert G \rVert_{\lebe^{p}(\ball_1(0)} \lesssim \lVert \nabla v \rVert_{\lebe^{p}(\ball_1(0)}^{p-1}\lVert H \rVert_{\lebe^{p}(\ball_1(0)}.
  \end{split}
\end{equation}
By using this a-priori estimate and regularising $G,$ unique solvability follows.
Similarly for~\ref{item:Lpbound2} we choose $\psi \in (W^{1,p'}_0 \cap W^{2,p'}(\ball_1(0),\mathbb R^{n^2N})$ to be the weak solution to the dual problem
\begin{align*}
\begin{cases} 
-\mathrm{div}(\mathcal{A}^*\nabla \psi ) = -\lvert\nabla^2w\rvert^{p-2}\nabla^2w & \text{in}\;\ball_{1}(0),\\ 
\psi= 0&\text{on}\;\partial\!\ball_{1}(0).
\end{cases} 
\end{align*}
Then we can estimate
\begin{equation}
\begin{split}
    & \int_{\ball_1(0)} \lvert\nabla^2 w\rvert^p \dif x = \int_{\ball_1(0)} \langle \nabla^2 w, -\mathrm{div}(\mathcal{A}^*\nabla\psi)\rangle \dif x = \int_{\ball_1(0)} \langle H I_n, \nabla^2 \psi\rangle \dif x \\
    &\quad\lesssim \lVert \nabla^2\psi \rVert_{\lebe^{p'}(\ball_1(0)}\lVert H \rVert_{\lebe^{p}(\ball_1(0)} \lesssim \lVert \nabla^2 w \rVert_{\lebe^{p}(\ball_1(0)}^{p-1}\lVert H \rVert_{\lebe^{p}(\ball_1(0)}.
\end{split}
\end{equation}
Finally~\ref{item:Lpbound3} is a standard interior regularity result; we refer to \cite[Proposition 2.10]{CFM} for an elementary proof.
\end{proof}

\begin{proof}[Proof of Remark {\ref{rem:gradientbounds}} for $p \geq 2$]
   Recalling that $p^{\#} = \frac{np}{n-1},$ we will also write $p_{\#} = \frac{np}{n+1}.$
  If $p \geq 2,$ using the same notation and rescaling as in the above proof we choose $\Phi$ to solve the auxiliary system
  \begin{equation}\label{eq:gradient_aux_pgeq2}
    \begin{cases} 
      -\mathrm{div}\big(\mathcal{A}\nabla\Phi\big) = -\mathrm{div}\big(\lvert \nabla \Psi\rvert^{s-2} \nabla\Psi\big)&\;\text{in}\;\ball,\\ 
      \Phi = 0&\;\text{on}\;\partial\!\ball, 
    \end{cases}
  \end{equation}
  where $s = 2_\{\#\} = \frac{2n}{n+1}.$ Then noting $\frac{1}{s-1} = \frac{n+1}{n-1},$ by~Lemma~\ref{lem:Lpbounds}~\ref{item:Lpbound1} we can estimate
  \begin{equation}\label{eq:gradient_aux_lpbound}
    \lVert \nabla\Phi\rVert_{\lebe^{\frac{np}{n-1}}(\ball)} \leq \lVert \nabla\Psi \rVert_{\lebe^{\frac{np}{n+1}}(\ball)}^{s-1}.
  \end{equation} 
  Hence proceeding similarly as in \eqref{eq:pgeq2_maincomparision} we get
  \begin{equation}
    \begin{split}
      \int_{\ball} \lvert\nabla \Psi\rvert^s \dif x 
      &\lesssim \int_{\ball} \omega_m(\lvert\nabla\widetilde U\rvert)\lvert\nabla\widetilde U\rvert^{p-1} \lvert\nabla\Phi\rvert \dif x \\
      &\leq \lVert\omega_m(\lvert\nabla\widetilde U\rvert) \rVert_{\lebe^{np}(\ball)} \lVert \nabla \widetilde U\rVert_{\lebe^p(\ball)}^{p-1} \lVert \nabla\Phi\rVert_{\lebe^{\frac{np}{n-1}}(\ball)} \\
      &\leq \rho_m\left( \int_{\ball} V_p(\nabla\widetilde U) \dif x \right)^{\frac{1}{np}} \left( \int_{\ball} V_p(\nabla\widetilde U) \dif x \right)^{\frac{p-1}p} \left( \int_{\ball} V_p(\nabla \Psi)^{\frac{n}{n+1}} \dif x \right)^{\frac{n-1}{np}}
    \end{split}
  \end{equation} 
  where $\rho_m = \omega_m \circ V_p^{-1}$ using Jensen's inequality, and we have used \eqref{eq:gradient_aux_lpbound} in the last line.
  Analogously to \eqref{eq:Vplarger2} we can estimate $\lVert \lvert\nabla\Psi\rvert\rVert_{\lebe^{\frac{np}{n+1}}(\ball)}^p \leq \left( \int_{\ball} V_p(\nabla\widetilde U) \dif x \right)^{\frac pq}$ for some $q > p$ if $p > 2,$ from which \eqref{eq:gradient_harmonixapprox} follows.
\end{proof}


\begin{thebibliography}{99}
\bibitem{AcerbiFusco}  Acerbi, E.; Fusco, N.: A regularity theorem for minimizers of quasiconvex integrals. Arch. Rational Mech. Anal.
99 (1987), no. 3, 261--281.
\bibitem{Allard} Allard, W.K.: On the first variation of a varifold. Ann. Math. (2) 95, 225--254 (1972).
\bibitem{AlibertDacorogna} Alibert, J.; Dacorogna, B.: An example of a quasiconvex function that is not polyconvex in two dimensions. Arch. Rational Mech. Anal., 117.2 (1992), 155--166.
\bibitem{Beck} Beck, L.: Elliptic Regularity - A first course. Lecture Notes of the Italian Mathematical Union (19), 2016, Springer.
\bibitem{BL} Bhattacharya, T; Leonetti, F: A new Poincar\'e inequality and its application to the regularity of minimizers of integral functionals with nonstandard growth. Nonlinear Anal., 17 (1991), no 9, 833--839.
\bibitem{CFM} Carozza, M.; Fusco, N.; Mingione, G.: Partial regularity of minimizers of quasiconvex integrals with subquadratic
growth. Ann. Mat. Pura Appl., IV, Ser. 175, 141--164 (1998). 

\bibitem{CP} Carozza, M.; Passarelli di Napoli, A.: A regularity theorem
for minimisers of quasiconvex integrals: the case $1<p<2$. Proc. Roy.
Soc. Edinburgh Sect. A 126 (1996), no. 6, 1181--1199.
\bibitem{DacorognaMarcellini} Dacorogna, B.; Marcellini, P.: A counterexample in the vectorial calculus of variations. Material Instabilities in Continuum Mechanics, (1988), 77--83
\bibitem{CK} Chen, C.Y.; Kristensen, J.: On coercive variational integrals. Nonlinear Anal., 153 (2017), 213--229.
\bibitem{DieningCruzzi} Cruz-Uribe, D.; Diening, L.: Sharp $\mathcal{A}$-harmonic approximation. Applicable Analysis, 98(1-2): 374-380, 2019.
\bibitem{DeGiorgi} De Giorgi, E.: Frontiere orientate di misura minima. Editrice Tecnico Scientifica (1961).
\bibitem{DLSV} Diening, L.; Lengeler, D.; Stroffolini, B.; Verde, A.: Partial regularity for minimizers of quasi-convex functionals
with general growth. SIAM J. Math. Anal. 44 (2012), no. 5, 3594--3616.
\bibitem{Duzaar1} Duzaar, F.; Grotowski, F.: Optimal interior partial regularity for nonlinear elliptic systems: the method of $\mathcal{A}$-
harmonic approximation. Manuscripta Math. 103 (2000), no. 3, 267--298.
\bibitem{Duzaar2} Duzaar, F.; Grotowski, J.F.; Kronz, M.: Regularity of almost minimizers of quasi-convex variational integrals with
subquadratic growth. Ann. Mat. Pura Appl. (4) 184 (2005), no. 4, 421--448.
\bibitem{DuzaarMingione} Duzaar, F.; Mingione, G.: Harmonic type approximation lemmas. J. Math. Anal. Appl. 352 (2009), no.1, 301--335.
\bibitem{Evans} Evans, L.C.:  Quasiconvexity and partial regularity in the calculus of variations. Arch. Rational Mech. Anal. 95
(1986), no. 3, 227--252.
\bibitem{FH} Fusco, N.; Hutchinson, J.: $C^{1,\alpha}$ partial regularity of functions minimising quasiconvex integrals, Mauscripta Math., 54 (1985), no.1, 121--143.
\bibitem{GM} Giaquinta, M.; Modica, G.: Partial regularity of minimizers of quasiconvex integrals, Ann. l'I.H.P. Anal. non lin\'eaire, 3 (1986), no.3, 185-208.
\bibitem{Giusti} Giusti, E.: Direct methods in the calculus of variations. World Scientific Publishing Co., Inc., River Edge, NJ, 2003.
\bibitem{G1} Gmeineder, F.: Partial regularity for symmetric quasiconvex functionals on BD. J. Math. Pures Appl. (9) 145 (2021), 83--129.
\bibitem{GK1}  Gmeineder, F.; Kristensen, J.: Partial regularity for BV minimizers. Arch. Rational Mech. Anal. 232 (2019) 1429--1473.
\bibitem{GK2} Gmeineder, F.; Kristensen, J.: Quasiconvex functionals of $(p,q)$-growth and the partial regularity of relaxed minimisers. ArXiv preprint, 2022.
\bibitem{Hamburger} Hamburger, C: Optimal partial regularity of minimizers of quasiconvex variational integrals, ESAIM Control Optim. Calc. Var., 13 (2007), 4, 639--656.
\bibitem{Irving} Irving, C.: Partial regularity for minima of higher-order quasiconvex integrands with natural Orlicz growth. Arxiv preprint, ArXiv 2111.14740.
\bibitem{KrasnoselskiiRutickii} Krasnosel'ski\u{\i}, M.A.; Ruticki\u{\i}, Y.B.: Convex functions and Orlicz spaces. P. Noordhoff, Groningen, 1961.
\bibitem{KristensenMingione} Kristensen, J.; Mingione, G.: The singular set of Lipschitzian minima of multiple integrals. Arch. Rational Mech. Anal., 184-2 (2007), 341--369.
\bibitem{Meyers} Meyers, N.: Quasi-convexity and lower semi-continuity of multiple variational integrals of any order. Trans. Amer. Math. Soc., 119 (1965), 125--149.
\bibitem{Mingione} Mingione, G.: Regularity of minima: an invitation to the dark side of the calculus of variations. Appl. Math. 51
(2006), no. 4, 355--426.
\bibitem{MullerSverak} M\"{u}ller, S.; \v{S}ver\'{a}k, V.: Convex integration for Lipschitz mappings and counterexamples to regularity, Ann. Math. 157 (2003), 715--742.
\bibitem{Simon1} Simon, L.: Lectures on geometric measure theory. Proceedings of the Centre for Mathematical Analysis, Australian National University, 3, Canberra, 1983.
\bibitem{Simon2} Simon, L.: Theorems on regularity and singularity of energy minimizing maps. Birkh\"auser, Basel, 1996.
\bibitem{Stampacchia} Stampacchia, G.: The spaces $L^{p,\lambda},$ $N^{p,\lambda}$ and interpolation. Ann. della Sc. Norm. Super. di Pisa, Cl. di Sci., 19 (1965), no.3, 443-462.
\end{thebibliography}
\end{document}